\documentclass{amsart}

\DeclareFontFamily{OML}{script}{}
\DeclareFontShape{OML}{script}{m}{it}
{ <5-20> rsfs10 }{}
\DeclareMathAlphabet{\mathscript}{OML}{script}{m}{it}

\renewcommand{\mathcal}[1]{{\mathscript #1}\hspace{0.2ex}}

\usepackage[pagewise]{lineno}
\usepackage{color}
\usepackage{float}
\ifx\red\undefined
\newcommand{\red}{\color{red}}

\fi

\makeatletter
\def\hlinewd#1{%
\noalign{\ifnum0=`}\fi\hrule \@height #1 \futurelet
\reserved@a\@xhline}
\makeatother

\usepackage{graphicx,graphics}
\usepackage{cite}
\usepackage{pifont}
\usepackage{amsthm}
\usepackage[final,notref]{showkeys}

\usepackage{amscd}
\usepackage{amsmath}
\usepackage{latexsym}
\usepackage{amsfonts}
\usepackage{amssymb}
\usepackage{color}
\usepackage{multicol}

\usepackage[ansinew]{inputenc}
\usepackage[encapsulated]{CJK}

\allowdisplaybreaks[4]

\newcommand{\re}[1]{\mbox{\rm$($\ref{#1}$)$}}
\newcommand{\dis}{\displaystyle}
\newcommand{\m}{\hspace{1em}}
\newcommand{\mm}{\hspace{2em}}
\newcommand{\p}{\partial}

\newcommand{\xx}{\vspace*{2ex}}

\renewcommand{\d}{\delta}

\renewcommand{\O}{\Omega}
\newcommand{\s}{\sigma}
\renewcommand{\epsilon}{\varepsilon}
\renewcommand{\t}{\widetilde}
\renewcommand{\theta}{\vartheta}
\renewcommand{\phi}{\varphi}

\ifx\text\undefined
\newcommand{\text}{\mbox}
\fi
\ifx\operatorname\undefined
\newcommand{\operatorname}{\mathop}
\fi

\newcommand\be{\begin{equation}}
\newcommand\ee{\end{equation}}
\newcommand\bea{\begin{eqnarray}}
\newcommand\eea{\end{eqnarray}}
\newcommand\beaa{\begin{eqnarray*}}
\newcommand\eeaa{\end{eqnarray*}}
\newcommand{\dif}{\mathrm{d}}

\newenvironment{eqa}{\begin{equation}%
  \begin{array}{rcl}}{\end{array}\end{equation}}

\newcommand\beqa{\begin{eqa}}
\newcommand\eeqa{\end{eqa}}

\numberwithin{equation}{section}

\renewcommand{\tilde}{\widetilde}
\renewcommand{\hat}{\widehat}
\renewcommand{\bar}{\overline}

\newtheorem{thm}{Theorem}[section]

\newtheorem{lem}{Lemma}[section]

\newtheorem{rem}{Remark}[section]

\newcommand{\void}[1]{}

\newcommand{\bR}{{\mathbb R}}

\numberwithin{equation}{section}

\begin{document}\begin{CJK}{UTF8}{gkai}
\title[Linear stability  for a multi-layer tumor model with time delay]{The linear stability  for a free boundary problem modeling multi-layer tumor growth with time delay
}
\author{Wenhua He }
\author{Ruixiang Xing }
\author{Bei Hu   }
\address{School of Mathematics, Sun Yat-sen University, Guangzhou 510275, China
}
\address{School of Mathematics, Sun Yat-sen University, Guangzhou 510275, China
}
\address{Department of Applied and Computational Mathematics and Statistics, University of Notre Dame, Notre Dame, Indiana 46556, USA}
\email{hewh27@mail2.sysu.edu.cn}
\email{xingrx@mail.sysu.edu.cn }
\email{b1hu@nd.edu
}

\maketitle

\begin{abstract}
We study a free boundary problem modeling multi-layer tumor growth with a small time delay $\tau$, representing the time needed for the cell to complete the replication process.  The model  consists of  two elliptic equations which describe  the concentration of nutrient and the tumor tissue pressure, respectively, an ordinary differential equation describing the cell location characterizing the  time delay and a partial differential equation for the free boundary. In this paper we establish the well-posedness of the problem, namely,  first we prove that there exists a unique  flat stationary solution  $(\sigma_*, p_*, \rho_*, \xi_* )$   for all $\mu>0$. The stability of this stationary solution should depend on the tumor aggressiveness constant $\mu$. It is also unrealistic to expect the perturbation to be flat. We show  that, under non-flat  perturbations,
there  exists a threshold  $\mu_*>0$ such that  $(\sigma_*, p_*, \rho_*, \xi_*)$  is linearly stable if $\mu<\mu_*$ and linearly unstable if  $\mu>\mu_*$. Furthermore, the time delay increases the
stationary tumor size. These are interesting results with  mathematical and biological implications.

 \xx\noindent

\xx\noindent

\end{abstract}

\section{Introduction}
There is a variety of shapes of tumors in tissue cultures. It is known that three-dimensional tumors  grown in tissue culture are  likely to take the shape of spheroids;  a large number of partial differential equation (PDE)  sphere-shaped tumors models have been developed, and a variety of properties including
well-posedness, asymptotic stability, bifurcation, the impact of a variety of biological   relevant parameters, etc., are studied. For example,  the first model of free boundary problem  for a solid tumor growth  is proposed and analyzed by  Greenspan  in \cite{1972Models}  and \cite{1976On}.   In
\cite{FR1},  Friedman and Reitich considered global well-posedness and global asymptotically stability for radially symmetric solutions.
For the non-symmetric case,  Bazaliy
and Friedman established the local well-posedness and asymptotic behavior under non-radial perturbations for the time-dependent problem in \cite{BVF} and \cite{BVF1}.  In particular, Friedman and Hu extended the work by  giving a precise threshold in \cite{FH4}.
 For more details,
 we refer to the papers \cite{ FH3, FH5,  Hongjing} and the references therein.


Medico-biologists have recently developed that cellular aggregates  gather on permeable membranes,  causing them to form multilayered tumor cell.
Because multilayered tumor cells are grown on permeable membranes which can separate two reservoirs of the diffusion apparatus directly, it   is an important task to study the fluidity  of drug and metabolism of tumor tissue.
See  \cite{2020Bifurcation,1997Three,2004Three,kyle1999characterization} for the study of multilayered tumor cells.

 Following the works of Cui and  Escher \cite{CE1} and Zhou,  Escher and Cui \cite{2008Bifurcation}, we consider in this paper the following   3-dimensional multilayered tumor   region of the  flat-shaped  form
$$
\Omega(t) \triangleq \{(x,y)\in \bR^{2}\times \bR; \;\;  0<y< \rho(t,x)\}, \mm {\bf x} = (x,y)= (x_1,x_2, y),
$$
where $\rho(t,x)$ is an unknown positive function. Denote by $\Gamma(t)$ the upper boundary
$\{y= \rho(t,x)\}$ of $\Omega(t)$ (the free boundary).
\begin{figure}[H] \label{bj}
\includegraphics[width=3.5in]{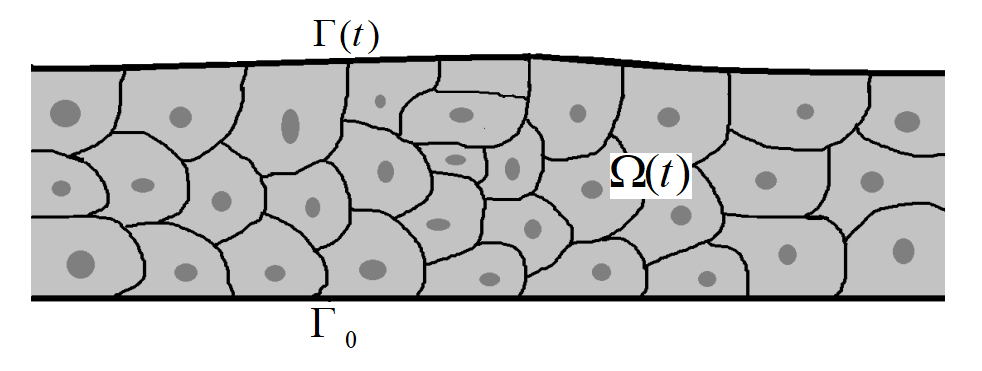}
\caption{}
\end{figure}
Through the upper boundary $\Gamma(t)$, a multi-layer tumor acquires
nutrients (denoted by $\sigma$), mostly oxygen or glucose, enabling tumor cells to grow and proliferate. The nutrient $\sigma$ satisfies the diffusion equation
$\lambda\sigma_t - \Delta\sigma + \sigma = 0$, where
 $\lambda$ is the ratio  of the rate for nutrients diffusion to the rate for the cell proliferation, so it is  small, and in this paper we assume a quasi-steady state approximation by taking $\lambda = 0$.

For simplicity, we assume that the tumor is immersed in an environment with nutrient concentration $\bar\sigma$. Let
$\Gamma_0$ denote the lower boundary $\{y=0\}$, which is assumed to be an impermeable layer, i.e.,
there is no nutrient flux through $\Gamma_0$:
\begin{eqnarray}
&&  - \Delta\sigma + \sigma = 0, \hspace{2em} (x_{1}, x_{2}, y)\in\Omega(t),\hspace{2em} t>0,\label{1.1}\\
&&\sigma  = \bar{\sigma}, \hspace{2em} (x_{1}, x_{2},y)\in \Gamma(t),\hspace{2em} t>0,\label{1.2}\\
&&\displaystyle\frac{\partial \sigma}{\partial y} =0, \hspace{2em} (x_{1}, x_{2},y)\in \Gamma_0, \hspace{2em} t>0. \label{1.3}
\end{eqnarray}

If the tumor is assumed to be of porous medium type where Darcy's law (i.e., $\vec{V}=-\nabla p$, where $p$ is the pressure, here we consider extracellular matrix as ``porous medium" in which cell moves) can be used, then the conversation of mass $\mbox{div} \vec V = S$ (where $S$ is the proliferation rate) implies
\[ -\Delta p = S. \]
The proliferation rate $S$ is proportional to $\sigma - \tilde \sigma$, where $\tilde \sigma$
is the threshold concentration that is needed by the tissue to maintain itself. Since the cells
need time (say $\tau$) to replicate and proliferate, it is assumed that  $S=  \mu[\sigma(\xi(t-\tau;{\bf x},t),t-\tau)-\tilde{\sigma}]$,  where $\mu$
is the tumor  aggressiveness constant and  $\xi(s; {\bf x},  t)$ tracks the
cell location at time $s$ which reaches the location ${\bf x}
=(x_1,x_2, y)$ at time $t$, and moves with the velocity field $\vec V = -\nabla p$:
\begin{eqnarray}  && \displaystyle\frac{\dif \xi(s;{\bf x}, t)}{\dif s}= -\nabla p( { \xi(s;{\bf x}, t),s)},\hspace{2em} t-\tau\le s\le t,\label{1.4}\\
&&\xi(s; x_{1}, x_{2},   y, t) = (x_{1}, x_{2}, y),\hspace{2em} s = t.\label{1.5}
\end{eqnarray}
Combining the expression of $S$ and the Darcy's law, we derive
\begin{equation}\label{1.5b}
-\Delta p =\mu[\sigma(\xi(t-\tau; x_{1}, x_{2},  y, t),t-\tau)-\tilde{\sigma}],\hspace{2em} (x_{1}, x_{2}, y)\in\Omega(t),\hspace{2em} t>0,
\end{equation}
and assuming the velocity field is continuous up to the boundary, the normal velocity of the moving boundary $\Gamma(t)$ is
\begin{equation}\label{1.5a}
V_n = -\nabla p\cdot n = -\frac{\partial p}{\partial n},  \hspace{2em} (x_{1}, x_{2},y)\in \Gamma(t),\hspace{2em} t>0.
\end{equation}

 Because most of the proteins and lipids that make up the cell membrane are held together
with the cell-to-cell adhesiveness,
we   have the boundary condition, see \cite{BC2},
\begin{eqnarray}\label{1.6}
&& p = \kappa , \hspace{2em} (x_{1}, x_{2},y)\in \Gamma(t),\hspace{2em} t>0,
\end{eqnarray}
where $\kappa$ is the mean curvature.  And
\begin{eqnarray}\label{1.6a}
&& \displaystyle\frac{\partial p}{\partial y} =0, \hspace{2em} (x_{1}, x_{2},y)\in \Gamma_0, \hspace{2em} t>0. \label{1.9a}
\end{eqnarray}
For convenience of our discussion, we shall also impose the $2\pi$-periodic condition in the $x_1$ and $x_2$ directions.

We finally prescribe initial conditions.  For simplicity
we assume initial  data are time independent on the interval $[-\tau,0]$:
\begin{eqnarray}\label{1.8}
&& \Omega(t)  = \Omega_0, \quad -\tau\le t \le 0, \\
\label{1.9}
&&  p (x_{1}, x_{2},y,t) = p_0 (x_{1}, x_{2},y), \quad {  (x_{1}, x_{2},y)\in \Omega_0,}\quad -\tau\le t \le 0,
\end{eqnarray}
where we assume the compatibility condition $\frac{\p p_0}{\p n} =0$ on $\p\O_0$.  The $p$ and $\xi$ are interdependent on the interval $[t-\tau, t]$;  the value of $\xi$ at $0$, for example, depends on the value of $p$ at $[-\tau, 0]$.  Once   the initial data
for $p$ is available on { $[-\tau,0]$,} we can solve $\xi$. So we only assume initial data for $p_0$.

The idea of adding time delay on the tumor model was initiated by Byrne \cite{delay2}, and recently, the radially symmetric version has drawn considerable attention of other researchers, see \cite{delay1, delay3, delay4, delay5, delay6}. The time delay represents the time taken for cells to undergo replication (approximately 24 hours). The non-radially symmetric model  was established by Zhao and Hu \cite{zhao, zhao2}, a radially symmetric stationary solution was found,  stability with respect to non-radially symmetric perturbation was studied, and bifurcation branches were established.  In this paper we shall extend the linear stability results to the
flat domains with non-flat perturbations. We begin with the existence and uniqueness of the stationary solution.  In contrast to the results in \cite{zhao}, our domain is different, resulting various distinct estimates need in order to carry out the proofs. {\em The stationary solution $(\sigma_*, p_*, \rho_*, \xi_*)$ is said to be {\em flat}  if $\sigma_*, p_*, \rho_*$ are independent of the variables $x_1, x_2$ and $\xi_*(s; x_{1}, x_{2},y) = (x_1, x_2,   \xi_{30}(s_*;y))$; roughly speaking, here $s_*$ represents the limit of the variable $s-t$ as $t\to \infty$ and therefore $-\tau\le s_* \le 0$: this is the amount of time needed to replace the dead cells by the same amount of new born cells   to make the tumor stationary.
}

\begin{thm}\label{T1.1}
For  all $\mu>0$, there exists a unique flat stationary  solution  $(\sigma_*, p_*, \rho_*, \xi_*)$ to the\textbf{} problem \re{1.1}-\re{1.9} for sufficiently small $\tau$.
\end{thm}

In order to obtain the linear stability results, we first  linearize the  system at the flat stationary solution  $(\sigma_*, p_*, \rho_*, \xi_*)$.

Assume the initial conditions are perturbed
from the stationary solution:
\begin{eqnarray}
&&\partial \Omega(t): y = \rho_* + \epsilon\rho_0(x_{1}, x_{2}),\hspace{2em} -\tau\le t\le0,\nonumber\\
&&  p(x_{1}, x_{2},y,t) = p_*(y) + \epsilon q_0(x_{1}, x_{2},y), \mm -\tau\le t\le 0.
\end{eqnarray}

Substituting
  \begin{eqnarray}
    & &\partial \Omega(t): y = \rho_* + \epsilon \rho(x_{1}, x_{2},t)+O(\epsilon^2),\nonumber\\
    &&\sigma(x_{1}, x_{2},y, t) = \sigma_*(y) + \epsilon w(x_{1}, x_{2},y,t)+O(\epsilon^2), \nonumber\\
    &&p(x_{1}, x_{2},y,t)= p_*(y) + \epsilon q(x_{1}, x_{2},y,t)+O(\epsilon^2), \nonumber\\
   & &{   \xi(s;x_{1}, x_{2},y,t)= \xi_*(s-t; x_{1}, x_{2},y)+\epsilon(\xi_{11}, \xi_{21}, \xi_{31})+O(\epsilon^2)}\nonumber
\end{eqnarray}
into \re{1.1}-\re{1.9} and collecting the $\epsilon$-order terms, we get the  linearized system for   $(\partial \Omega, \sigma, p, \xi)$ at the flat stationary solution $(\sigma_*, p_*, \rho_*, \xi_*)$.
We define
\bea
    \mu_{j}(\rho_*^0)&=&\frac{  \displaystyle \frac12 j^{3/2} \tanh(\sqrt{j}\rho_*^0)}{ \bar\s \; k_1({j},\rho_*^0)} \hspace{2em}\text{ for } ~ {j}> j_0.\\
  k_1(j, \rho_*^0) & = & 1-\frac{\tanh \rho_*^0}{\rho_*^0}
 -\tanh \rho_*^0\cdot \Big[\sqrt{1+j}\tanh(\sqrt{1+j}\rho_*^0 )-
 \sqrt{j}\tanh(\sqrt{j}\rho_*^0)\Big] ,
\eea
where $\rho_*^0$ is the zeroth-order terms in $\tau$ of $\rho_*$ and $j_0$ is the unique zero of $k_1(\cdot,\rho_*^0)$.
Setting
\be
\mu_{j}(\rho_*^0) = +\infty \m\text{for } 0\le j\le j_0, \mm
\mu_*(\rho_*^0) =   \min_{j > j_0}\mu_{j}(\rho_*^0).
\ee
We now state the linear stability result of the flat stationary solution $(\sigma_*, p_*, \rho_*, \xi_* )$.
\begin{thm}\label{T1.2}
For sufficiently small $\tau$,
there exists
a threshold value $\mu_*(\rho_*^0)>0$ such that
the stationary solution $(\sigma_*, p_*, \rho_*,  \xi_*)$  is linearly stable if $\mu < \mu_*(\rho_*^0)$, i.e.,
there exist $C>0$ and $\delta >0$ such that for the problem linearized  in both
$\epsilon$-perturbation terms and in time-delay $\tau$, respectively,
\begin{equation}
    |\rho(t)| \le Ce^{-\delta t} \text{ for all } t>0,
\end{equation}
the stationary solution $(\sigma_*, p_*, \rho_*,  \xi_*)$  is linearly unstable if $\mu > \mu_*(\rho_*^0)$.

\end{thm}
The structure of this article is as follows.  In section 2, we  collect some properties of hyperbolic function which will be useful later. We prove the existence and uniqueness of a flat stationary solution by using the contraction mapping principle in section 3. In section 4, we obtain the linearized system of  \re{1.1}-\re{1.9} and establish the linear stability results. We show the impact of time delay for tumor growth in section 5  and present mathematical and biological implications of our results in section 6.

\section{Preliminaries}
For convenience, we collect some elementary properties for special functions which are needed
 later on.

 The following are easy to verify:
\begin{eqnarray}
&&\label{2.1}
\frac{d}{d \rho}\frac{\tanh \rho}{\rho}=\frac{1}{ \rho}\Big(1-\frac{\tanh \rho}{\rho}-\tanh^{2}\rho\Big)=\displaystyle\frac{\rho-\sinh \rho\cosh \rho}{\rho^2\cosh^{2} \rho}<0,\hspace{2em} \rho>0,\\
&&\label{2.2}
\lim_{\rho\rightarrow0} \frac{\tanh \rho}{\rho}=1,\hspace{2em}\lim_{\rho\rightarrow+\infty} \frac{\tanh \rho}{\rho}=0,\\
&&
\label{2.3}
 \int e^{\sqrt{j}x} \cosh(\sqrt{1+j}x)  \dif x=\sqrt{1+j}e^{\sqrt{j}x}\sinh(\sqrt{1+j}x)
-\sqrt{j}e^{\sqrt{j}x}\cosh(\sqrt{1+j}x),
\\&&
\label{2.4}
 \int e^{-\sqrt{j}x} \cosh(\sqrt{1+j}x)  \dif x=\sqrt{1+j}e^{-\sqrt{j}x}\sinh(\sqrt{1+j}x)
+\sqrt{j}e^{-\sqrt{j}x}\cosh(\sqrt{1+j}x),\\ &&
\label{2.5}
  \int e^{-\sqrt{j}x}\sinh(\sqrt{1+j}x)  \dif x=\sqrt{1+j}e^{-\sqrt{j}x}\cosh(\sqrt{1+j}x)
+\sqrt{j}e^{-\sqrt{j}x}\sinh(\sqrt{1+j}x),
\end{eqnarray}
and
\begin{eqnarray}
  &&\label{2.6}\int_0^{\rho}\sinh^{2} y\dif y=\frac{1}{2}\sinh\rho\cosh\rho-\frac{1}{2}\rho,\\
  &&\label{2.7}
  \int_0^{\rho}y\sinh y \dif y=\rho\cosh\rho-\sinh\rho.
\end{eqnarray}

It is also derived in \cite[section 4]{CE1},
\begin{equation}\label{2.8}
    \frac{\partial^{2} }{\partial x^{2}}[\sqrt{x}\tanh (\sqrt{x}\rho)]
=\dis\frac{\rho}{2}\frac{\partial}{\partial x}\;\frac{\sinh(\sqrt{x}\rho)\cosh(\sqrt{x}\rho)+\sqrt{x}\rho}{(\rho\sqrt{x})\cosh^{2}(\sqrt{x}\rho)}<0,
\end{equation}
for $x>0$ and $\rho>0$.

\section{Flat Stationary Solution}
In this section, we prove  that there exists a unique  flat stationary solution $(\sigma_*, p_*, \rho_*,  \xi_*)$  of  the system  \re{1.1}-\re{1.9} for all $\mu>0$. Whenever there is no confusion, it is customary to  let $C$ to  denote various
 positive constants in our estimates, although it may change from one line to another.
Letting the $t$-derivatives to be zero in \re{1.1}--\re{1.9}, we   find that the  stationary problem  is of the form
\begin{eqnarray}
&&\label{3.1}\left \{
\begin{array}{lr}
-\sigma''(y) + \sigma(y) = 0, \hspace{2em} 0<y< \rho,\\
\sigma(\rho) = \overline{\sigma}, \mm
\displaystyle\frac{\partial \sigma}{\partial y}\Big|_{y=0}=0,
\end{array}
\right.\\
&&\label{3.2}\left \{
\begin{array}{lr}
-  p''(y) = \mu[\sigma (\xi_{30}(-\tau;y))-\tilde{\sigma}], \hspace{2em} 0<y< \rho, \\
p(\rho) = 0,\mm
\displaystyle\frac{\partial p}{\partial y}\Big|_{y=0}=0,
\end{array}
\right.\\
&&\label{3.3}\left \{
\begin{array}{lr}
\displaystyle
\frac{\dif  \xi_{30}}{\dif s_*}(s_*;y) = -\frac{\partial p}{\partial y} (\xi_{30}(s_*;y)), \hspace{2em} -\tau\le s_*\le 0,\\
 \xi_{30}(s_*;y) = y, \hspace{2em}\hspace{2em}\hspace{2em}\hspace{2em} s_* = 0,\\
\end{array}
\right.
\\
&&\label{3.4}\int_0^{\rho} \Big(\sigma( \xi_{30}(-\tau;y)) - \tilde{\sigma}\Big)\dif y = 0.
\end{eqnarray}
The equation \re{3.1} admits an explicit solution:
\begin{equation}\nonumber
\sigma_*(y) =\overline{ \sigma}\frac{\cosh y}{\cosh\rho}.
\end{equation}
We now proceed to establish the existence of a unique flat stationary  solution $(\sigma_*, p_*, \rho_*,  \xi_*)$ to the problem \re{1.1}-\re{1.9}.

{\bf \noindent Proof of Theorem \ref{T1.1}. }
Taking $\hat{y} = \displaystyle\frac{y}{\rho}$, $\hat{\sigma}(\hat{y})=\sigma(y)$, $\hat{p}(\hat{y})=\rho p(y)$ and $\hat{\xi_{30}}(s_*;\hat{y}) =\displaystyle \frac{\xi_{30}(s_*;y)}{\rho}$
into \re{3.1}--\re{3.4},
dropping the $``~\hat{}~"$ for { notational convenience}, we get
\begin{eqnarray}
&&\label{3.5}\left \{
\begin{array}{lr}
\sigma''(y) =\rho^2 \sigma(y) , \hspace{2em} 0<y<1,\\
\sigma(1) = \overline{\sigma},\mm
\displaystyle\frac{\partial \sigma}{\partial y}\Big|_{y=0}=0,
\end{array}
\right.\\
&&\label{3.6}\left \{
\begin{array}{lr}
- p''(y) = \mu\rho^3\Big[\sigma\Big(y+\displaystyle\frac{1}{\rho^3}\int_{-\tau}^0\displaystyle\frac{\partial p}{\partial y}(\xi_{30}(s_*;y)) \dif s \Big)-\tilde{\sigma}\Big], \hspace{2em} 0<y<1,\\
p(1) = 0,\mm
\displaystyle\frac{\partial p}{\partial y}\Big|_{y=0}=0,
\end{array}
\right.\\
&&\label{3.7}\left \{
\begin{array}{lr}
\displaystyle
\frac{\dif \xi_{30}}{\dif s_*}(s_*;y) = -\displaystyle\frac{1}{\rho^3}\frac{\partial p}{\partial y}(\xi_{30}(s_*;y)), \hspace{2em} -\tau\le s_*\le 0,\hspace{2em} 0<y<1,\\
\xi_{30}(s_*;y) = y, \hspace{2em}\hspace{2em}\hspace{2em}\hspace{2em} s_* = 0,\\
\end{array}
\right.
\\
&&\int_0^{1} \Big[\sigma\Big(y+\displaystyle\frac{1}{\rho^3}\int_{-\tau}^0\displaystyle\frac{\partial p}{\partial y}(\xi_{30}(s_*;y)) \dif s \Big) - \tilde{\sigma}\Big]\dif y = 0.\label{3.8}
\end{eqnarray}

Equation \re{3.5} is solved explicitly. For convenience, we also extend the solution outside $[0,1]$:
\begin{equation}\label{3.9}
\sigma_*(y;\rho) =\overline{ \sigma}\frac{\cosh(\rho y)}{\cosh\rho},
\m 0\le y \le 1, \mm \bar\sigma_*(y;\rho) = \bar \sigma, \m  1<y\le 2.
\end{equation}

Assume that $\rho_*$ exists and will be in the range of its maximum value $\rho_{\max}$ and minimum value $\rho_{\min}$ which will be  determined  later on.
By integrating the first equation of $\re{3.6}$, we have
\begin{equation}\label{q}
p(y) = \int_y^{1}\int_0^{\eta}\mu\rho_*^3\Big[\sigma_*\Big(z+\displaystyle\frac{1}{\rho_*^3}\int_{-\tau}^0\displaystyle\frac{\partial p}{\partial y}(\xi_{30}(s_*;z)) \dif s \Big)-\tilde{\sigma}\Big] \dif z  \dif \eta.
\end{equation}
Next we prove the existence and uniqueness of $p$ by using the contraction mapping principle. Obviously, $0$  is the lower solution of \re{3.7}. But there is no guarantee that the solution of \re{3.7} stays
below the upper boundary $\{y=1\}$. So we shall proceed  as in \cite{zhao}  to extend $p$ beyond $y=1$.  Let
$$X =  \{p\in W^{2,\infty}[0,2]; \|p\|_{W^{2,\infty}[0,2]}
\le 3\mu \rho^3_{\max} (\bar\sigma +\tilde{\sigma})\}.$$
For each $p\in X$, we first solve $\xi_{30}$ from the ODE \re{3.7},
and substitute it into \re{q} to define
 a mapping $T$:
\begin{equation}
    \label{T}
   Tp(y) = \displaystyle\int_y^{1}\int_0^{\eta}\mu\rho_*^3\Big[\sigma_*\Big(z+\displaystyle\frac{1}{\rho_*^3}\int_{-\tau}^0\displaystyle\frac{\partial p}{\partial y}(\xi_{30}(s_*;z)) \dif s \Big)-\tilde{\sigma}\Big] \dif z  \dif \eta, \mm 0\le y\le 1.
\end{equation}
Clearly, $Tp(1)=0, \displaystyle\frac{\partial( Tp)}{\partial y}\Big|_{y=0}=0$. We now extend $Tp$ to the interval $[0,2]$ by defining
\begin{equation}
    \label{extend}
    Tp(y) =
    \left\{
    \begin{split}
        &Tp(y),\hspace{2em} &0\leq y\le 1,\\
        & Tp'(1)(y-1),\hspace{2em} &1<y\le 2.
    \end{split}
    \right.
\end{equation}
It is clear with this extension,  $Tp$ is continuous with continuous derivative across $y=1$, and $Tp\in W^{2,\infty}[0,2]$.
\void{
obtain a unique solution $\t p_{k+1}\in X$ of the following system:
\begin{eqnarray}\ \label{3.10}
\left \{
\begin{array}{lr}
p_{k+1}(y) = \displaystyle\int_y^{1}\int_0^{y}\mu\rho_*^3[\sigma_*(z+\displaystyle\frac{1}{\rho_*^3}\int_{-\tau}^0\displaystyle\frac{\partial p_{k}}{\partial y}(\xi_{k}(s;z,0)) \dif s )-\tilde{\sigma}] \dif z  \dif y,\\
p_{k+1}(1) = 0,\\
\displaystyle\frac{\partial p_{k+1}}{\partial y}\Big|_{y=0}=0.
\end{array}
\right.
\end{eqnarray}
Define a mapping $T$:
\begin{equation}
    \label{T}
   Tp_{k}(y) = \displaystyle\int_y^{1}\int_0^{y}\mu\rho_*^3[\sigma_*(z+\displaystyle\frac{1}{\rho_*^3}\int_{-\tau}^0\displaystyle\frac{\partial p_{k}}{\partial y}(\xi_{k}(s;z,0)) \dif s )-\tilde{\sigma}] \dif z  \dif y.
\end{equation}
}

Using the expression in \re{T} and
the extension \re{extend}, estimating respectively on the interval $[0,1]$ and $[1,2]$, we find  that
\be
\|\displaystyle Tp\big\|_{W^{2,\infty}[0,2]}
\le 3\mu \rho^3_{\max} (\bar\sigma +\tilde{\sigma}),
\ee
 and therefore $T$ maps $X$ into itself.

We shall establish that $T$ is a contraction, namely, for some $M<1$,

\begin{equation}
    \label{X}
   \|T\t p -T p \|_{X}\leq M \|\t p-p\|_{X}, \ \  \ \forall \  \t p, \   p\in X.
\end{equation}

Next, we prove \re{X}.
Let $\xi_{30}$ and $\t\xi_{30}$ be the corresponding solutions.   By integrating the first equation of $\re{3.7}$, we have
\begin{equation*}
    \begin{split}
        \max_{\substack{-\tau\le s_*\le 0\\0\le y\le 1}}|\t\xi_{30}(s_*;y,0) - \xi_{30}(s_*;y,0)| &= \max_{\substack{-\tau\le s_*\le 0\\0\le y\le 1}}\bigg|\displaystyle\frac{1}{\rho_*^3}\int_{s_*}^0\bigg[\displaystyle\frac{\partial \t p}{\partial y}(\t\xi_{30}(s_*;y))-\displaystyle\frac{\partial  p}{\partial y}(\xi_{30}(s_*;y)) \bigg]\dif s_*\bigg|\\
        &\le  \frac{\tau}{\rho_*^3}\bigg[ \|\t p -  p\|_{W^{2,\infty}[0,2]} + \| p\|_{W^{2,\infty}[0,2]}
         \max_{\substack{-\tau\le s_*\le 0\\0\le y\le 1}}|\t\xi_{30} - \xi_{30}|\bigg]\\
        &\le  \frac{\tau}{\rho_*^3}\bigg[ \|\t p -  p\|_{W^{2,\infty}[0,2]} + C
         \max_{\substack{-\tau\le s_*\le 0\\0\le y\le 1}}|\t\xi_{30} - \xi_{30}|\bigg],
    \end{split}
\end{equation*}
where   by the choice of our $X$,  $\| p\|_{W^{2,\infty}[0,2]}\le  3\mu \rho^3_{\max} (\bar\sigma +\tilde{\sigma})\triangleq C<\frac{\rho_*^3}{\tau}$
if $\tau$ is small. Thus
\begin{equation}
    \label{3.18}
    \max_{\substack{-\tau\le s_*\le 0\\0\le y\le 1}}|\t\xi_{30}(s_*;y,0) - \xi_{30}(s_*;y,0)| \le \frac{\tau}{\rho_*^3 - \tau C}\|\t p -  p\|_{W^{2,\infty}[0,2]}.
\end{equation}
From \re{3.18}, \re{extend} and \re{T},
 \begin{equation}
\begin{split}
    \label{3.22}
    &\|(T\t p-T p)''\|_{L^\infty[0,2]}\\
    &=\bigg\|\mu\rho_*^3\sigma_*\Big(y+\displaystyle\frac{1}{\rho_*^3}\int_{-\tau}^0\displaystyle\frac{\partial \t p}{\partial y}(\t\xi_{30}) \dif s_* \Big) - \mu\rho_*^3\sigma_*\Big(y+\displaystyle\frac{1}{\rho_*^3}\int_{-\tau}^0\displaystyle\frac{\partial  p}{\partial y}(\xi_{30}) \dif s_* \Big) \bigg\|_{L^\infty[0,1]} \\
   &\le \mu \rho^3_{\max}\Big\|\frac{\partial \sigma_*}{\partial y}\Big\|_{L^\infty[0,2]}\frac{\tau}{\rho^3_{\min}-C\tau}\|\t p -  p\|_{W^{2,\infty}[0,2]}.
  \end{split}
\end{equation}
Since $T p(1)=T \t p(1)=0$ and $(Tp)'(0)=(T\t p)'(0)=0$, the above estimates
imply
\be
\|T\t p-T p\|_{W^{2,\infty}[0,2]}
 \le  C \mu \rho^3_{\max}\Big\|\frac{\partial \sigma_*}{\partial y}\Big\|_{L^\infty[0,2]}\frac{\tau}{\rho^3_{\min}-C\tau}\|\t p -  p\|_{W^{2,\infty}[0,2]}.
\ee
If $\tau$ is suitably small, then $M \triangleq  C \mu \rho^3_{\max}\Big\|\frac{\partial \sigma_*}{\partial y}\Big\|_{L^\infty[0,2]}\frac{\tau}{\rho^3_{\min}-C\tau}<1$,  therefore
we established \re{X} and $T$ is a
contraction, which admits a unique fixed point  $p_*$. Substituting $p_*$ into
\re{3.7} and  from ODE  theory, we obtain $ \xi_*$.

To complete the proof, it suffices to show that there exists a unique solution $\rho_* \in [\rho_{\min}, \rho_{\max}]$  satisfying \re{3.8}.  After substituting \re{3.9} into \re{3.8}, we find that this is equivalent to solving the following equation for $\rho$:
$$
F(\rho,\tau) \triangleq \int_0^1 \Bigg\{ \overline{\sigma}\frac{\cosh\Big[\rho\Big(y+\displaystyle\frac{1}{\rho^3}\int_{-\tau}^0\displaystyle\frac{\partial p}{\partial y}(\xi_{30}(s_*;y)) \dif s_*\Big)\Big]}{\cosh\rho}- \tilde{\sigma}\Bigg\} \dif y = 0.$$

Clearly,
$$F(\rho,0)= \int_0^1 \Big(\overline{\sigma}\frac{\cosh(\rho y)}{\cosh\rho}- \tilde{\sigma}\Big) \dif y=\frac{\overline{\sigma}}{\rho}\tanh\rho-\tilde{\sigma},$$
and from \re{2.1} and \re{2.2},
$$\lim_{\rho\rightarrow0} F(\rho,0)=\overline{\sigma}-\widetilde{\sigma}>0, \ \ \ \lim_{\rho\rightarrow\infty} F(\rho,0)=-\widetilde{\sigma}<0. $$

Notice that \re{2.1} also implies that  $F(\rho,0)$ is monotone decreasing in $\rho$,  so that the equation $F(\rho,0) = 0$ admits
a unique solution (denoting by $\rho_S$) and
\be \label{rhoS}
F\Big(\frac{1}{2} \rho_S, 0\Big) > 0, \hspace{1em} F\Big(\frac{3}{2} \rho_S, 0\Big) < 0.
\ee
The mean value theorem implies, for some $0\leq\eta\leq\tau$,
\begin{equation}\nonumber
\frac{\partial F(\rho,\tau)}{\partial \rho}- \frac{\partial F(\rho,0)}{\partial \rho}=\frac{\partial^2F}{\partial \rho\partial \tau}(\rho,\eta)\tau
=O(\tau) .
\end{equation}
It follows that $\frac{\partial F(\rho,\tau)}{\partial \rho}<0$
when $\tau$ is small enough. In a similar argument, we also have
$$ F\Big(\frac{1}{2} \rho_S, \tau\Big) > 0, \hspace{1em} F\Big(\frac{3}{2} \rho_S, \tau\Big) < 0.$$
Therefore, when $\tau$ is small enough, the equation \re{3.8} admits
a unique solution $\rho_*$ satisfying $F(\rho_*,\tau)=0$ and $\frac12 \rho_S < \rho_* < \frac32 \rho_S$. The proof is complete with $\rho_{\min} = \frac12 \rho_S$ and $\rho_{\max} = \frac32 \rho_S$.
\qed\xx

\section{Linear Stability}
In this section, we consider the linear stability of the unique flat stationary  solution
$(\sigma_*,p_*,\rho_*,  \xi_*)$ obtained in section 3 under non-flat perturbations. We begin by taking some small non-flat perturbations on the initial conditions:
\begin{eqnarray}
&&\partial \Omega(t): y = \rho_* + \epsilon\rho_0(x_{1}, x_{2}),\hspace{2em} -\tau\le t\le0,\label{4.1}\\
&&{   p (x_{1}, x_{2},y,t) ={  p_*(y) + \epsilon }p_0 ( x_{1}, x_{2},y), \quad{ (x_{1}, x_{2},y)\in\Omega_0},\quad -\tau\le t \le 0} \label{4.2}.
\end{eqnarray}
Then for $t>0$, we expect to have formal expansion:
\begin{eqnarray}
     & &\partial \Omega(t): y = \rho_* + \epsilon \rho(x_{1}, x_{2},t)+O(\epsilon^2), \nonumber\\
    &&\sigma(x_{1}, x_{2},y,t) = \sigma_*(y) + \epsilon w(x_{1}, x_{2},y,t)+O(\epsilon^2), \label{4.3}\\
    &&p(x_{1}, x_{2},y,t)= p_*(y) + \epsilon q(x_{1}, x_{2},y,t)+O(\epsilon^2), \nonumber  \\
    && \xi(s;x_{1}, x_{2},y,t)=\xi_*(s-t;x_1,x_2,y)+\epsilon(\xi_{11}, \xi_{21}, \xi_{31})+O(\epsilon^2).\nonumber
\end{eqnarray}
Writing in Cartesian coordinates,
\begin{eqnarray}
     &&\xi (s;x_{1}, x_{2},y,t)=\xi_1(s;x_{1}, x_{2},y,t)\overrightarrow{i}+\xi_2(s;x_{1}, x_{2},y,t)\overrightarrow{j}+\xi_3(s;x_{1}, x_{2},y,t) \overrightarrow{k},\label{4.4}
\end{eqnarray}
we obtain from \re{1.4}--\re{1.5} that
\begin{equation}\label{4.7}
\left \{
\begin{split}
&\frac{\dif \xi_1}{\dif s} (s;x_{1}, x_{2},y,t)= -\frac{\partial p}{\partial x_{1}}(\xi_1,\xi_2,\xi_3,s),\quad t-\tau\le s\le t,\\
&\xi_1 (s;x_{1}, x_{2},y,t)\Big|_{s=t} =x_{1} ;
\end{split}
\right.
\end{equation}
\begin{equation}\label{4.8}
\left \{
\begin{split}
&\frac{\dif \xi_2}{\dif s} (s;x_{1}, x_{2},y,t)= -\frac{\partial p}{\partial x_{2}}(\xi_1,\xi_2,\xi_3,s),\quad t-\tau\le s\le t,\\
&\xi_2 (s;x_{1}, x_{2},y,t)\Big|_{s=t} =x_{2} ;
\end{split}
\right.
\end{equation}
\begin{equation}\label{4.9}
\left \{
\begin{split}
&\frac{\dif \xi_3}{\dif s} (s;x_{1}, x_{2},y,t)= -\frac{\partial p}{\partial y}(\xi_1,\xi_2,\xi_3,s),\quad t-\tau\le s\le t,\\
&\xi_3 (s;x_{1}, x_{2},y,t)\Big|_{s=t} =y.
\end{split}
\right.
\end{equation}
We then expand  $\xi_1, \xi_2, \xi_3$ in $\epsilon$ as
\begin{eqnarray}
     & &\xi_1 (s;x_{1}, x_{2},y,t)= x_{1} + \epsilon \xi_{11}(s;x_{1}, x_{2},y,t) + O(\epsilon^2),\nonumber\\
    &&\xi_2(s;x_{1}, x_{2},y,t) =  x_{2}+ \epsilon \xi_{21} (s;x_{1}, x_{2},y,t)+ O(\epsilon^2), \label{4.10}\\
    &&\xi_3(s;x_{1}, x_{2},y,t) =  \xi_{30}(s-t;y) + \epsilon \xi_{31} (s;x_{1}, x_{2},y,t)+ O(\epsilon^2). \nonumber
\end{eqnarray}
Recalling that we already obtained the zeroth order equation  $\frac{\dif \xi_{30}}{\dif s_*} = - \frac{\partial p_*}{\partial y}{  (\xi_{30}(s_*;y))}$, $\xi_{30}\Big|_{ {s_*=0}} = y$ (cf. \re{3.3}).
By substituting \re{4.10} into \re{4.7}--\re{4.9} and dropping the higher order terms, we  obtain the first order equations for $\xi$:
\begin{eqnarray}
&&\label{4.12}\left \{
\begin{array}{lr}
\dis\frac{\dif \xi_{11}}{\dif s} =  - \dis\frac{\partial q}{\partial x_{1}}{ (x_{1}, x_{2},  \xi_{30}, s)},\hspace{2em} t-\tau\le s \le t,\\
\xi_{11}\Big|_{s=t} = 0;
\end{array}
\right.
\end{eqnarray}
\begin{eqnarray}
&&\label{4.14}\left \{
\begin{array}{lr}
\dis\frac{\dif \xi_{21}}{\dif s}=  - \dis\frac{\partial q}{\partial  x_{2}}{  (x_{1}, x_{2}},\xi_{30}, s),
\hspace{2em} t-\tau\le s \le t,\\
\xi_{21}\Big|_{s=t} = 0;
\end{array}
\right.
\end{eqnarray}
\begin{eqnarray}
&&\label{4.16}\left \{
\begin{array}{lr}
\dis\frac{\dif \xi_{31}}{\dif s} = -\dis\frac{\partial^2 p_*}{\partial y^2}{ (\xi_{30}(s-t;y))}\xi_{31}(s;x_{1}, x_{2},y,t) - \dis\frac{\partial q}{\partial y}{  (x_{1}, x_{2}, \xi_{30},s)},
\hspace{2em} t-\tau\le s \le t,\\
\xi_{31}\Big|_{s=t} = 0.
\end{array}
\right.
\end{eqnarray}

By substituting \re{4.3} and \re{4.12}--\re{4.16} into \re{1.1}--\re{1.9}, applying  the following  mean-curvature
formula in the 3-dimensional case for $y=f(x_1,x_2)$:
\[
 \kappa = -\frac{(1+f_{x_{2}}^2)f_{x_{1}x_{1}}+(1+f_{x_{1}}^2)f_{x_{2}x_{2}}-2f_{x_{1}}f_{x_{2}}f_{x_{1}x_{2}}}{2\big(1+f_{x_{1}}^2+f_{x_{2}}^2\big)^{3/2}},
\]
and collecting the $\epsilon$-order terms, we get the linearized system of \re{1.1}--\re{1.9}:
\begin{eqnarray}
&&\label{4.17}\left \{
\begin{array}{lr}
\Delta w(x_{1}, x_{2},y,t) = w(x_{1}, x_{2},y,t),  \hspace{2em}0<y<\rho_{\ast}, \hspace{2em}t>0,\\
w(x_{1}, x_{2},y,t)\Big|_{y=\rho_{\ast}}=-\dis\frac{\partial \sigma_*}{\partial y}\bigg|_{y=\rho_*}\rho(x_{1}, x_{2},t),\m
\dis\frac{\partial w}{\partial y}(x_{1}, x_{2},y,t)\Big|_{y=0}=0,
\end{array}
\right.\\
&&\label{4.18}\left \{
\begin{array}{lr}
- \Delta q(x_{1}, x_{2},y,t) = \mu\dis\frac{\partial \sigma_*}{\partial y}{{  (\xi_{30}(-\tau;y))}}\xi_{31}(t-\tau;x_{1}, x_{2},y,t)\\
 \ \ \ \ \ \ \ \ \ \ \ \ \ \ \ \ \ \ \ \ \ \ \ \ \ +\mu w{  (x_{1}, x_{2},\xi_{30}(-\tau;y)}, t-\tau),  \hspace{2em}0<y<\rho_{\ast}, \hspace{2em}t>0,\\
q(x_{1}, x_{2},y,t)\Big|_{y=\rho_{\ast}}=-\dis \frac12
(\rho_{x_{1}x_{1}}+\rho_{x_{2}x_{2}}),\m
\dis\frac{\partial q}{\partial y}(x_{1}, x_{2},y,t)\Big|_{y=0}=0,
\end{array}
\right.\\
&&\dis\frac{\dif \rho}{\dif t}(x_{1}, x_{2},t) = -\dis\frac{\partial^2 p_*}{\partial y^2}\bigg|_{y=\rho_*}\rho(x_{1}, x_{2},t) - \dis\frac{\partial q}{\partial y}\bigg|_{y=\rho_*}(x_{1}, x_{2},y,t).\label{4.19}
\end{eqnarray}
 Together with  \re{4.16}, we obtain a linearized system.
We look for solutions of the form:
\begin{eqnarray*}
&& w(x_{1}, x_{2},y,t) =
w_{n,m}(y,t)\cos(nx_{1})\cos(m x_{2}),\\
&&  q(x_{1}, x_{2},y,t) =
  q_{n,m}(y,t)\cos(nx_{1})\cos(m x_{2}),\\
&& \rho(x_{1}, x_{2},t) =
\rho_{n,m}(t)\cos(nx_{1})\cos(m x_{2}),\\
&&  \xi_{31}(s;x_{1}, x_{2},y,t) =
  \varphi_{n,m}(s;y,t)\cos(nx_{1})\cos(m x_{2}).
\end{eqnarray*}
As we shall easily  verify that the equations for $w_{n,m}, q_{n,m}, \rho_{n,m}, \phi_{n,m}$ will not change
if we replace $\cos(n x_1)\cos(m x_2)$ by any of the following
\[
 \cos(n x_1)\sin(m x_2), \m \sin(n x_1)\cos(m x_2), \m \sin(n x_1)\sin(m x_2).
 \]
 These constitute a base for the
 Fourier series $2\pi$ periodic in $x_1$ and $x_2$.

From \re{4.17}--\re{4.19}, we derive
\begin{eqnarray}
&&\label{4.20}\left \{
\begin{array}{lr}
- \dis\frac{\partial^2w_{n,m}}{\partial y^2}(y,t) +(n^{2}+m^{2}+1)w_{n,m}(y,t) = 0,\\

w_{n,m}(\rho_*,t) = -\dis\frac{\partial \sigma_*}{\partial y}\Big |_{y=\rho_*} \rho_{n,m}(t),\m
\dis\frac{\partial w_{n,m}}{\partial y}(0,t) =0,
\end{array}
\right.\\
&&\label{4.21}\left \{
\begin{array}{lr}
-\dis\frac{\partial^2 q_{n,m}}{\partial y^2}(y,t)+ (n^{2}+m^{2})q_{n,m}(y,t) = \mu w_{n,m}{    (\xi_{30}(-\tau;y)},t-\tau)\\
 \ \ \ \ \ \ \ \ \ \ \ \ \ \ \ \ \ \ \ \ \ \ \ \ \ \ \ \ \ \ \ \ \ \ \ \ \ \ \ \ \ \ \ \ \ \ \ \  \ \ + \mu\dis\frac{\partial \sigma_*}{\partial y}{   (\xi_{30}(-\tau;y))}\varphi_{n,m}(t-\tau;y,t),\\
 q_{n,m}(\rho_*,t) =\dis \frac12(n^{2}+m^{2}) \rho_{n,m}(t),\m
\dis\frac{\partial q_{n,m}}{\partial y}(0,t) =0 ,
\end{array}
\right.\\
&&\label{4.22}\left \{
\begin{array}{lr}
\dis\frac{\partial \varphi_{n,m}}{\partial s}(s;y,t) = -\dis\frac{\partial^2 p_*}{\partial y^2}{  (\xi_{30}(s-t;y))}\varphi_{n,m}(s;y,t) - \dis\frac{\partial q_{n,m}}{\partial y}{  (\xi_{30}(s-t;y),s)}, \\  \ \ \ \ \ \ \ \ \ \ \ \ \ \ \ \ \ \ \ \ \ \ \ \ \ \ \ \ \ \ \ \ \ \ \ \ \ \ \ \ \ \ \ \ \ \ \ \  \ \  \ \ \ \ \ \ \ \ \ \ \ \ \ \ \ \ \ \ \ \ \ \ \ \ \ \ \ \ \ \ \ \ \ \ \ \ \ \  t-\tau\le s \le t,\\
\varphi_{n,m}(s;y,t)\Big|_{s=t} = 0,
\end{array}
\right.\\
&&\dis\frac{\dif \rho_{n,m}}{\dif t} (t)= -\dis\frac{\partial^2 p_*}{\partial y^2}\Big |_{y=\rho_*} \rho_{n,m}(t) - \frac{\partial q_{n,m}}{\partial y}\Big |_{y=\rho_*}.\label{4.23}
\end{eqnarray}
Solving \re{4.20}, we obtain
\begin{equation}\label{4.24}
 w_{n,m} =-\frac{\partial \sigma_*}{\partial y}(\rho_*) \rho_{n,m}(t)\frac{\cosh(\sqrt{n^{2}+m^{2}+1}y)}{\cosh(\sqrt{n^{2}+m^{2}+1}\rho_*)}.
\end{equation}

\subsection{Expansion in $\tau$.} As $\tau$ is small, we seek expansion in $\tau$ of the form:
\begin{eqnarray*}
&&\rho_*=\rho_*^0 + \tau \rho_*^1 + O(\tau^2),\\
&&\sigma_* (y)=\sigma_*^0(y) + \tau\sigma_*^1(y) + O(\tau^2),\\
&&p_*(y) = p_*^0(y) +  \tau p_*^1 (y) + O(\tau^2),\\
&&w_{n,m}(y,t) = w_{n,m}^0(y,t) + \tau w_{n,m}^1(y,t) + O(\tau^2),\\
&&q_{n,m}(y,t) = q_{n,m}^0 (y,t)+ \tau q_{n,m}^1(y,t) + O(\tau^2),\\
&&\rho_{n,m}(t) = \rho_{n,m}^0(t) + \tau\rho_{n,m}^1(t)  + O(\tau^2).
\end{eqnarray*}
\subsubsection{Expansion of \re{3.1}.} From \re{3.1}, we find
$$\sigma_*^0 + \tau\sigma_*^1 + O(\tau^2)= \bar{\sigma}\frac{\cosh y}{\cosh \rho_*}=  \bar{\sigma}\Big\{\frac{\cosh y}{\cosh \rho_*^0} -\tau \frac{\cosh y \sinh \rho_*^0 }{\cosh^{2}\rho_*^0} \rho_*^1+ O(\tau^2)\Big\},$$
therefore,
\begin{eqnarray}
&&\sigma_*^0(y) = \bar{\sigma}\frac{\cosh y}{\cosh \rho_*^0}\label{4.25},\hspace{2em}\sigma_*^1(y) = -\bar{\sigma}\rho_*^1 \frac{\sinh \rho_*^0  }{\cosh^{2}\rho_*^0}\cosh y .\label{4.26}
\end{eqnarray}
The boundary conditions in \re{3.1} are expanded as follows:
$$\sigma_*^0(\rho_*^0) + \tau\frac{\partial \sigma_*^0}{\partial y}(\rho_*^0)\rho_*^1 + \tau \sigma_*^1(\rho_*^0) + O(\tau^2) = \overline{\sigma}, \mm$$
Thus,
\begin{eqnarray}\nonumber
&&\sigma_*^0(\rho_*^0)  = \overline{\sigma},\hspace{2em}\displaystyle\frac{\partial \sigma_*^0 }{\partial y}\Big|_{y=0}=0,\\\nonumber
&&\sigma_*^1(\rho_*^0) = -\frac{\partial \sigma_*^0}{\partial y}(\rho_*^0)\rho_*^1,\hspace{2em}
\displaystyle\frac{\partial \sigma_*^1 }{\partial y}\Big|_{y=0}=0.
\end{eqnarray}
\subsubsection{Expansion of \re{3.2}.}
Integrating equation \re{3.3} over the interval $(-\tau,0)$, we derive
\begin{equation}\label{4.27}
 \xi_{30}(-\tau;y)  = y + \int_{-\tau}^0 \frac{\partial p_*}{\partial y}({  \xi_{30}(s;y))}\dif s =
 y + \tau \frac{\partial p_*^0}{\partial y} (y)+ O(\tau^2).
 \end{equation}
It follows that
\begin{equation}\nonumber
    \begin{split}
    \sigma_*({  \xi_{30}(-\tau;y)) }&= \sigma_*\bigg(y + \tau \frac{\partial p_*^0}{\partial y} (y)+ O(\tau^2)\bigg)
        &= \sigma_*^0(y) + \tau\Big(\frac{\partial \sigma_*^0}{\partial y}(y)\frac{\partial p_*^0}{\partial y}(y) +  \sigma_*^1(y)\Big) + O(\tau^2).    \end{split}
\end{equation}
Substituting the above expression into the right-hand side of \re{3.2}, we get
\begin{gather}\nonumber
-\frac{\partial^2 p_*^0}{\partial y^2} = \mu(\sigma_*^0 - \tilde{\sigma}), \hspace{2em}
-\frac{\partial^2 p_*^1}{\partial y^2}  = \mu\frac{\partial \sigma_*^0}{\partial y}\frac{\partial p_*^0}{\partial y}+\mu\sigma_*^1.
\end{gather}
The boundary conditions in \re{3.2} are expanded as follows:
$$
    p_*^0(\rho_*^0) + \tau\frac{\partial p_*^0}{\partial y}(\rho_*^0)\rho_*^1 + \tau p_*^1(\rho_*^0) + O(\tau^2) = 0, \mm
$$
Therefore
\begin{eqnarray}\nonumber
&&p_*^0(\rho_*^0)  = 0,\hspace{2em}\displaystyle\frac{\partial p_*^0 }{\partial y}\Big|_{y=0}=0,\hspace{2em}
p_*^1(\rho_*^0) = -\frac{\partial p_*^0}{\partial y}(\rho_*^0)\rho_*^1,\hspace{2em}
\displaystyle\frac{\partial p_*^1 }{\partial y}\Big|_{y=0}=0.
\end{eqnarray}
\subsubsection{Expansion of \re{3.4}.} From \re{4.27}, we have
\begin{equation}\label{4.28}
    \begin{split}
      0 & =  \int_0^{\rho_*}[\sigma_*({  \xi_{30}(-\tau;y))}-\tilde{\sigma}]
      \dif y\\
        &= \int_0^{\rho_*}[\sigma_*^0(y)-\tilde{\sigma}]\dif y + \tau \int_0^{\rho_*^0}\bigg[\frac{\partial \sigma_*^0}{\partial y}(y)\frac{\partial p_*^0}{\partial y}(y)+\sigma_*^1(y)\bigg]\dif y + O(\tau^2).
    \end{split}
\end{equation}
The first part of \re{4.28} can be integrated as follows:
\begin{equation}\label{4.29}
    \begin{split}
      \int_0^{\rho_*} [\sigma_*^0(y)-\tilde{\sigma}]\dif y&=  \int_0^{\rho_*}\bigg(\bar{\sigma}\frac{\cosh y}{\cosh \rho_*^0}-\tilde{\sigma}\bigg)\dif y
        = \bar{\sigma}\frac{\sinh \rho_*}{\cosh \rho_*^0}-\tilde{\sigma}\rho_*\\
        &=\bar{\sigma}\tanh\rho_*^0-\tilde{\sigma}\rho_*^0+\tau(\bar{\sigma}\rho_*^1-\tilde{\sigma}\rho_*^1)+ O(\tau^2).
    \end{split}
\end{equation}
By combining \re{4.28} and \re{4.29}, we deduce
\begin{gather}\nonumber
    \bar{\sigma}\tanh\rho_*^0- \tilde{\sigma}\rho_*^0= 0,\\\nonumber
     (\bar{\sigma} -\tilde{\sigma})\rho_*^1+ \int_0^{\rho_*^0}\bigg[\frac{\partial \sigma_*^0}{\partial y}(y)\frac{\partial p_*^0}{\partial y}(y)+\sigma_*^1(y)\bigg]\dif y= 0.
\end{gather}
Recalling  the definitions of $\rho_s$ in \re{rhoS} and $F(\rho, 0)$ in the proof of Theorem \ref{T1.1}, we must have $\rho_*^0 = \rho_S$.
\subsubsection{Expansion of \re{4.20}.} The expansion of the first equation in \re{4.20} is straightforward  we omit it. The boundary conditions are evaluated as follows:
\begin{equation}
\begin{split}\nonumber
&w_{n,m}^0(\rho_*^0)+\tau\frac{\partial w_{n,m}^0}{\partial y}(\rho_*^0)\rho_*^1 + \tau w_{n,m}^1(\rho_*^0)\\
&= - \bigg(\frac{\p\sigma_*^0}{\p y}(\rho_*^0) +\tau\frac{\p^{2}\sigma_*^0}{\p y^{2}}( \rho_*^0)\rho_*^1+
\tau \frac{\p\sigma_*^1}{\p y}(\rho_*^0) \bigg)\bigg [\rho_{n,m}^0(t) + \tau \rho_{n,m}^1(t) \bigg] + O(\tau^2),
\end{split}
\end{equation}
which implies
\bea\nonumber
 && w_{n,m}^0(\rho_*^0,t) = -\frac{\partial \sigma_*^0}{\partial y}(\rho_*^0)\rho_{n,m}^0(t), \\\nonumber
 && w_{n,m}^1(\rho_*^0,t) = -\frac{\partial w_{n,m}^0}{\partial y}(\rho_*^0,t)\rho_*^1-\frac{\partial \sigma_*^0}{\partial y}(\rho_*^0)\rho_{n,m}^1(t) - \frac{\partial^2 \sigma_*^0}{\partial y^2}(\rho_*^0)\rho_*^1 \rho_{n,m}^0(t) -\frac{\partial \sigma_*^1}{\partial y}(\rho_*^0)\rho_{n,m}^0(t).
\eea
Similarly, the boundary condition at $\{y=0\}$ is evaluated:
$$
\displaystyle\frac{\partial w_{n,m}^0 }{\partial y}\Big|_{y=0}+\tau\displaystyle\frac{\partial w_{n,m}^1 }{\partial y}\Big|_{y=0}+ O(\tau^2)=0,
$$
which gives
$$\frac{\partial  w_{n,m}^0}{\partial y}|_{y=0}=0,\hspace{2em}\frac{\partial  w_{n,m}^1}{\partial y}|_{y=0}=0.$$
\subsubsection{Expansion of \re{4.21}.} To find the expansion of \re{4.21}, we first compute, by \re{4.22}:
\begin{equation}\nonumber\begin{split}
\varphi_{n,m}(t-\tau;y,t) =& \varphi_{n,m}(t;y,t) - \tau\frac{\partial \varphi_{n,m}}{\partial s}(t;y,t) + O(\tau^2)\\
 =&   0 -\tau \bigg(-\frac{\partial^2 p_*}{\partial y^2}(y)\varphi_{n,m}(t;y,t)-\frac{\partial q_{n,m}}{\partial y}(y,t)\bigg) + O(\tau^2)\\
=& \tau\frac{\partial q_{n,m}^0}{\partial y}(y,t) + O(\tau^2);
\end{split}
\end{equation}
it follows that
\begin{equation}\label{4.30}
    \begin{split}
\frac{\partial \sigma_*}{\partial y} (\xi_{30}(-\tau;y))\varphi_{n,m}(t-\tau;y,t)
        &= \bigg(\frac{\partial \sigma_*^0}{\partial y}(y) + O(\tau)\bigg)\bigg(\tau \frac{\partial q_{n,m}^0}{\partial y}(y,t) + O(\tau^2)\bigg)\\
        &= \tau \frac{\partial \sigma_*^0}{\partial y}(y)\frac{\partial q_{n,m}^0}{\partial y}(y,t) + O(\tau^2).
    \end{split}
\end{equation}
On the other hand,
\begin{equation}\label{4.31}
\begin{split}
 & \hspace{-3em} w_{n,m} (\xi_{30}(-\tau;y),t-\tau)\\
    =&w_{n,m}\Big({ \xi_{30}(0;y)}-\tau\frac{\partial \xi_{30}}{\partial s}(0;y)+O(\tau^2),t-\tau\Big)\\
     =&w_{n,m}\Big(y+\tau\frac{\partial p_*}{\partial y}(\xi_{30}(0;y))+O(\tau^2),t-\tau\Big) \mm { (\text{ by }  \re{3.3} })\\
      =&w_{n,m}\Big(y+\tau\frac{\partial p_*^0}{\partial y}(y)+O(\tau^2),t-\tau\Big)\\
      =&w_{n,m}(y,t)+\tau\frac{\partial w_{n,m}}{\partial y}(y,t)\frac{\partial p_*^0}{\partial y}(y)-\tau\frac{\partial w_{n,m}}{\partial t}(y,t)+ O(\tau^2)\\
    =& w_{n,m}^0(y,t) + \tau\bigg[\frac{\partial w_{n,m}^0}{\partial y}(y,t)\frac{\partial p_*^0}{\partial y}(y) - \frac{\partial w_{n,m}^0}{\partial t}(y,t) + w_{n,m}^1(y,t) \bigg] + O(\tau^2).
    \end{split}
\end{equation}
By substituting  \re{4.30}--\re{4.31} into \re{4.21}, we obtain
\begin{gather}\nonumber
    -\frac{\partial^2 q_{n,m}^0}{\partial y^2}+ (n^{2}+m^{2})q_{n,m}^0 = \mu w_{n,m}^0,\\\nonumber
    -\frac{\partial^2 q_{n,m}^1}{\partial y^2}+ (n^{2}+m^{2})q_{n,m}^1 = \mu \frac{\partial \sigma_*^0}{\partial y}\frac{\partial q_{n,m}^0}{\partial y}+ \mu\frac{\partial w_{n,m}^0}{\partial y}\frac{\partial p_*^0}{\partial y} -\mu\frac{\partial w_{n,m}^0}{\partial t} + \mu w_{n,m}^1.
\end{gather}
The first boundary condition in \re{4.21} is given by
\[
q_{n,m}^0(\rho_*^0,t)+\tau \frac{\partial q_{n,m}^0}{\partial y} (\rho_*^0,t) \rho_*^1+\tau q_{n,m}^1 (\rho_*^0,t)= \dis \frac12(n^{2}+m^{2})[\rho_{n,m}^0(t) + \tau \rho_{n,m}^1(t) ] + O(\tau^2),
\]
which implies
\bea\nonumber
 &&  q_{n,m}^0(\rho_*^0,t) = \dis \frac12(n^{2}+m^{2})\rho_{n,m}^0(t),  \\\nonumber
 && q_{n,m}^1(\rho_*^0,t)=-\frac{\partial q_{n,m}^0}{\partial y}(\rho_*^0,t)\rho_*^1 + \dis \frac12(n^{2}+m^{2})\rho_{n,m}^1(t).
\eea
Likewise, the second boundary condition is evaluated as:
$$
\displaystyle\frac{\partial q_{n,m}^0 }{\partial y}\Big|_{y=0}+\tau\displaystyle\frac{\partial q_{n,m}^1 }{\partial y}\Big|_{y=0}+ O(\tau^2)=0,
$$
which gives
$$\frac{\partial  q_{n,m}^0}{\partial y}\Big|_{y=0}=0,\hspace{2em}\frac{\partial  q_{n,m}^1}{\partial y}\Big|_{y=0}=0.$$
\subsubsection{Expansion of \re{4.23}.}
The following equation
\bea
\begin{split}\nonumber
 \frac{\dif}{\dif t}[\rho_{n,m}^0(t)+\tau \rho_{n,m}^1(t)] =&
 - \bigg(\frac{\p^2 p_*^0}{\p y^2}(\rho_*^0+\tau \rho_*^1) +
\tau \frac{\p^2 p_*^1}{\p y^2}(\rho_*^0) \bigg) [\rho_{n,m}^0(t) + \tau \rho_{n,m}^1(t) ]\\
&- \frac{\p( q_{n,m}^0+\tau q_{n,m}^1)}{\p y}(\rho_*^0+\tau \rho_*^1) + O(\tau^2)
\end{split}
\eea
implies
\bea\nonumber
&& \frac{\dif \rho_{n,m}^0(t)}{\dif t}  = -\frac{\partial ^2 p_*^0}{\partial y^2}(\rho_*^0)\rho_{n,m}^0(t) - \frac{\partial q_{n,m}^0}{\partial y}(\rho_*^0,t), \\\nonumber
&&     \begin{split}
    \frac{\dif \rho_{n,m}^1(t)}{\dif t}= &-\frac{\partial^2 p_*^0}{\partial y^2}(\rho_*^0)\rho_{n,m}^1(t)-\frac{\partial^3 p_*^0}{\partial y^3}(\rho_*^0)\rho_*^1 \rho_{n,m}^0(t)\\
    &-\frac{\partial^2 p_*^1}{\partial y^2}(\rho_*^0)\rho_{n,m}^0(t)-\frac{\partial^2 q_{n,m}^0}{\partial y^2}(\rho_*^0,t)\rho_*^1 - \frac{\partial q_{n,m}^1}{\partial y}(\rho_*^0,t).
    \end{split}
\eea

\subsection{zeroth-order terms in $\tau$} Collecting the zeroth-order terms from subsection 4.1, we obtain
\begin{eqnarray}
&&\label{4.32}\sigma_*^0(y) = \bar{\sigma}\frac{\cosh y}{\cosh \rho_*^0},\\
&&\label{4.33}\left \{
\begin{array}{lr}
-   (p_*^0)''(y) = \mu[\sigma_*^0(y)-\tilde{\sigma}], \hspace{2em} 0<y< \rho,\\
p_*^0(\rho^0_*) = 0,\m
\displaystyle\frac{\partial p_*^0}{\partial y}(0)=0,
\end{array}
\right.\\
&&\label{4.34'}\frac{\tanh \rho^0_*}{\rho^0_*}=\frac{\tilde{\sigma}}{\overline{\sigma}}, \mm i.e., \;\; \rho_*^0 = \rho_S,\\
&&\label{4.35}\left \{
\begin{array}{lr}
- \displaystyle\frac{\partial^2w^{0}_{n,m}}{\partial y^2}(y,t) +(n^{2}+m^{2}+1)w^{0}_{n,m}(y,t) = 0,\\
w^{0}_{n,m}(\rho^{0}_*,t) = -\displaystyle\frac{\partial \sigma^{0}_*}{\partial y}\Big |_{y=\rho^{0}_*} \rho^{0}_{n,m}(t),\m
\displaystyle\frac{\partial w^{0}_{n,m}}{\partial y}(0,t)=0,
\end{array}
\right.
\\
&&\label{4.36}\left \{
\begin{array}{lr}
-\displaystyle\frac{\partial^2 q^{0}_{n,m}}{\partial y^2}(y,t)+ (n^{2}+m^{2})q^{0}_{n,m}(y,t) = \mu w^{0}_{n,m}(y,t),\\
q^{0}_{n,m}(\rho^{0}_*,t) =\dis \frac12(n^{2}+m^{2}) \rho^{0}_{n,m}(t),\m
\displaystyle\frac{\partial q^{0}_{n,m}}{\partial y}(0,t)=0,\\
\end{array}
\right.
\\
&&\frac{\dif \rho_{n,m}^0(t)}{\dif t} = -\frac{\partial ^2 p_*^0}{\partial y^2}(\rho_*^0)\rho_{n,m}^0(t) - \frac{\partial q_{n,m}^0}{\partial y}(\rho_*^0,t).\label{4.37}
\end{eqnarray}
From \re{4.32}, we obtain
\bea\label{4.38}
 &&  \frac{\partial\sigma_*^0}{\partial y}(y)  = \overline{\sigma}\frac{\sinh y}{\cosh \rho_*^0},\hspace{2em}
\frac{\partial^{2} \sigma_*^0}{\partial y^{2}}(y)  =\overline{\sigma}\frac{\cosh y}{\cosh \rho_*^0}.
\eea
Substituting \re{4.32} into  \re{4.33}, we solve $p^0_*$ explicitly:
\begin{equation}
    p_*^0(y) = \frac{1}{2}\mu \tilde{\sigma} y^2 +\mu\overline{\sigma}- \frac{1}{2}\mu\tilde{\sigma}(\rho_*^0)^2-\mu \overline{\sigma}\frac{\cosh y}{\cosh \rho_*^0};
\end{equation}
it follows that
\bea
 && \label{4.39} \frac{\partial p_*^0}{\partial y}(y)  = \mu \tilde{\sigma} y-\mu \overline{\sigma}\frac{\sinh y}{\cosh \rho_*^0},\hspace{2em} \frac{\partial^{2} p_*^0}{\partial y^{2}}(y)  = \mu \tilde{\sigma} -\mu \overline{\sigma}\frac{\cosh y}{\cosh \rho_*^0}.
\eea
Similarly, \re{4.35} is solved explicitly:
\begin{equation}\label{4.43}
    w_{n,m}^0(y,t) = -\frac{\partial  \sigma_*^0}{\partial y}(\rho_*^0)\frac{\cosh (\sqrt{1+n^{2}+m^{2}}y)}{\cosh (\sqrt{1+n^{2}+m^{2}} \rho_*^0)}\rho_{n,m}^0(t).
\end{equation}
To solve $q^{0}_{n,m}$, we let
\begin{equation}\label{4.44}
z(y,t)=\displaystyle\frac{\partial q^{0}_{n,m}}{\partial y}(y,t)-\sqrt{n^{2}+m^{2}}q^{0}_{n,m}(y,t).
\end{equation}
Then
$$
\frac{\p z}{\p y}(y,t)+\sqrt{n^{2}+m^{2}}z(y,t)=- \mu w^{0}_{n,m}(y,t) = \mu\frac{\partial  \sigma_*^0}{\partial y}(\rho_*^0)\frac{\cosh \sqrt{1+n^{2}+m^{2}}y}{\cosh \sqrt{1+n^{2}+m^{2}} \rho_*^0}\rho_{n,m}^0(t).
$$
Using the integration formula \re{2.3}, we obtain
\begin{equation}\label{4.45}
    \begin{split}
        z(y,t)&=e^{-\sqrt{n^{2}+m^{2}}y}\bigg[\displaystyle\int \mu\displaystyle\frac{\partial  \sigma_*^0}{\partial y}(\rho_*^0)\rho_{n,m}^0(t)\frac{\cosh (\sqrt{1+n^{2}+m^{2}}y)}{\cosh (\sqrt{1+n^{2}+m^{2}} \rho_*^0)} e^{\sqrt{n^{2}+m^{2}}y} \dif y +C_{1}\bigg]\\
         &=\mu\displaystyle\frac{\partial  \sigma_*^0}{\partial y}(\rho_*^0)\rho_{n,m}^0(t)\frac{\sinh( \sqrt{1+n^{2}+m^{2}} y)}{\cosh (\sqrt{1+n^{2}+m^{2}} \rho_*^0)}\sqrt{1+n^{2}+m^{2}}\\
 & \ \ \ -\mu\displaystyle\frac{\partial  \sigma_*^0}{\partial y}(\rho_*^0)\rho_{n,m}^0(t)\frac{\cosh (\sqrt{1+n^{2}+m^{2}}y) }{\cosh (\sqrt{1+n^{2}+m^{2}} \rho_*^0)}\sqrt{n^{2}+m^{2}}+C_{1}e^{-\sqrt{n^{2}+m^{2}}y},
    \end{split}
\end{equation}
where $C_{1}$ is independent of $y$.

From \re{4.44} and integration formula \re{2.4}--\re{2.5}, we get
\begin{equation}\label{4.46}
    \begin{split}
       q^{0}_{n,m}&=e^{\sqrt{n^{2}+m^{2}}y}\bigg(\displaystyle\int e^{-\sqrt{n^{2}+m^{2}}y}z\dif y +C_{2}\bigg)\\
&=\mu\displaystyle\frac{\partial  \sigma_*^0}{\partial y}(\rho_*^0)\rho_{n,m}^0(t)\frac{\cosh (\sqrt{1+n^{2}+m^{2}}y )}{\cosh (\sqrt{1+n^{2}+m^{2}} \rho_*^0)}-\displaystyle\frac{C_{1}}{2\sqrt{n^{2}+m^{2}}}e^{-\sqrt{n^{2}+m^{2}}y}+C_{2}e^{\sqrt{n^{2}+m^{2}}y}.
    \end{split}
\end{equation}
Applying the boundary conditions in \re{4.36}, we find
\begin{eqnarray}
&&\label{4.47}
\begin{array}{lr}
C_{1}=\displaystyle\frac{\bigg[2\mu\dis\frac{\partial  \sigma_*^0}{\partial y}(\rho_*^0)-(n^{2}+m^{2})\bigg]\sqrt{n^{2}+m^{2}}}{2\cosh (\sqrt{n^{2}+m^{2}}\rho_*^0)}\rho_{n,m}^0(t),\;
C_{2}=\displaystyle\frac{\dis \frac12(n^{2}+m^{2})-\mu\dis\frac{\partial  \sigma_*^0}{\partial y}(\rho_*^0)}{2\cosh (\sqrt{n^{2}+m^{2}}\rho_*^0)}\rho_{n,m}^0(t).
\end{array}
\end{eqnarray}
Substituting the expressions $C_1$ and $C_2$ into \re{4.46}, we derive
\begin{eqnarray}\nonumber
q^{0}_{n,m}
=\rho_{n,m}^0(t)\Bigg\{\mu\displaystyle\frac{\partial  \sigma_*^0}{\partial y}(\rho_*^0) \frac{\cosh (\sqrt{1+n^{2}+m^{2}}y) }{\cosh (\sqrt{1+n^{2}+m^{2}} \rho_*^0)}+ \displaystyle\frac{(n^{2}+m^{2})-2\mu\displaystyle\frac{\partial  \sigma_*^0}{\partial y}(\rho_*^0)}{2\cosh (\sqrt{n^{2}+m^{2}}\rho_*^0)}\cosh (\sqrt{n^{2}+m^{2}}y) \Bigg\}.
\end{eqnarray}
Hence
\bea\label{4.48}
\begin{split}
\frac{\partial q^{0}_{n,m}}{\partial y}(\rho_*^0)
=&\mu\displaystyle\frac{\partial  \sigma_*^0}{\partial y}(\rho_*^0)\rho_{n,m}^0(t)\sqrt{1+n^{2}+m^{2}}\tanh(\sqrt{1+n^{2}+m^{2}} \rho_*^0)\\
&+\dis\frac12\rho_{n,m}^0(t)\bigg(n^{2}+m^{2}-2\mu\dis\frac{\partial  \sigma_*^0}{\partial y}(\rho_*^0)\bigg)\sqrt{n^{2}+m^{2}}\tanh (\sqrt{n^{2}+m^{2}}\rho_*^0).
\end{split}
\eea
Combining \re{4.39} and \re{4.48} with \re{4.37}, we get
\begin{equation}\label{rho}
\frac{\dif \rho_{n,m}^0(t)}{\dif t} =h_{n,m}(\mu,\rho_*^0)\rho_{n,m}^0(t),
\end{equation}
where,  after substituting $\frac{\p \s_*^0}{\p y}(\rho_*^0)=\bar\sigma\tanh \rho_*^0$
and $\tilde\sigma = \frac{\bar \sigma \tanh \rho_*^0}{\rho_*^0}$ into the expressions,
\bea\label{4.50}
\begin{split}
h_{n,m}(\mu,\rho_*^0)
=&\mu\bigg[   \overline{\sigma}- \widetilde{\sigma}-\displaystyle\frac{\partial  \sigma_*^0}{\partial y}(\rho_*^0)\sqrt{1+n^{2}+m^{2}}\tanh(\sqrt{1+n^{2}+m^{2}} \rho_*^0) \\
&+ \dis\frac{\partial  \sigma_*^0}{\partial y} (\rho_*^0)\sqrt{n^{2}+m^{2}}\tanh (\sqrt{n^{2}+m^{2}}\rho_*^0) \bigg]-\dis\frac12(n^{2}+m^{2})^{3/2}\tanh (\sqrt{n^{2}+m^{2}}\rho_*^0) \\
=& \mu \bar \sigma \bigg[1-\frac{ \tanh \rho_*^0}{\rho_*^0}
- \sqrt{1+n^{2}+m^{2}}\;\tanh \rho_*^0 \;\tanh(\sqrt{1+n^{2}+m^{2}} \rho_*^0)\\
&+  \sqrt{n^{2}+m^{2}}\;\tanh \rho_*^0 \;\tanh (\sqrt{n^{2}+m^{2}}\rho_*^0) \bigg]-\dis\frac12(n^{2}+m^{2})^{3/2}\tanh (\sqrt{n^{2}+m^{2}}\rho_*^0).
\end{split}
\eea

\subsection{first-order terms in $\tau$}Collecting first-order terms from subsection 4.1, we obtain the following system:
\be
\sigma_*^1(y) = -\bar{\sigma}\rho_*^1 \frac{\sinh \rho_*^0  }{\cosh^{2}\rho_*^0}\cosh y .
\ee
\vspace*{-.5em}
\bea
&&\label{4.75}\left \{
\begin{array}{lr}
-\dis\frac{\partial^2 p_*^1}{\partial y^2} = \mu\dis\frac{\partial \sigma_*^0}{\partial y}\dis\frac{\partial p_*^0}{\partial y}+\mu\sigma_*^1, \\
p_*^1(\rho_*^0)=-\dis\frac{\partial p_*^0}{\partial y}(\rho_*^0)\rho_*^1,\m
\dis\frac{\partial p_*^1}{\partial y}(0)=0,
\end{array}
\right.  \label{4.76}
\eea
\be
\rho_*^1(\overline{\sigma}-\widetilde{\sigma})+\int_0^{\rho_*^0}\Big(\frac{\partial \sigma_*^0}{\partial y}(y)\frac{\partial p_*^0}{\partial y}(y)+\sigma_*^1(y)\Big)\dif y = 0,
\label{4.45@}\ee
\be
\left \{\label{4.77}
\begin{array}{lr}
\dis\frac{\partial^2 w_{n,m}^1}{\partial y^2}- (n^{2}+m^{2}+1) w_{n,m}^1 = 0,\\
 w_{n,m}^1(\rho_*^0,t) = -\dis\frac{\partial w_{n,m}^0}{\partial y}(\rho_*^0,t)\rho_*^1-\dis\frac{\partial \sigma_*^0}{\partial y}(\rho_*^0)\rho_{n,m}^1(t) \\
  \ \ \ \ \ \ \ \ \ \ \ \ \ \ \ \ \ \ \  - \dis\frac{\partial^2 \sigma_*^0}{\partial y^2}(\rho_*^0)\rho_*^1 \rho_{n,m}^0(t) -\dis\frac{\partial \sigma_*^1}{\partial y}(\rho_*^0)\rho_{n,m}^0(t),\\
\dis\frac{\partial  w_{n,m}^1}{\partial y}(0,t)=0,
\end{array}
\right.
\ee
\be
\label{4.79}\left \{
\begin{array}{lr}
\dis\frac{\partial^2 q_{n,m}^1}{\partial y^2}-(n^{2}+m^{2}) q_{n,m}^1 = -\mu\dis\frac{\partial \sigma_*^0}{\partial y}\dis\frac{\partial q_{n,m}^0}{\partial y} - \mu\dis\frac{\partial w_{n,m}^0}{\partial y}\frac{\partial p_*^0}{\partial y}+\mu\dis\frac{\partial w_{n,m}^0}{\partial t} - \mu w_{n,m}^1,\\
    q_{n,m}^1(\rho_*^0,t)=-\dis\frac{\partial q_{n,m}^0}{\partial y}(\rho_*^0,t)\rho_*^1 + \dis\frac12(n^{2}+m^{2})\rho_{n,m}^1(t) \triangleq c_1(t),\m
    \dis\frac{\partial  q_{n,m}^1}{\partial y}(0,t)=0,
\end{array}
\right.
\ee
\bea
&&\label{4.80}
\frac{\dif \rho_{n,m}^1(t)}{\dif t}=-\frac{\partial^2 p_*^0}{\partial y^2}(\rho_*^0)\rho_{n,m}^1(t)-\frac{\partial^3 p_*^0}{\partial y^3}(\rho_*^0)\rho_*^1 \rho_{n,m}^0(t)-\frac{\partial^2 p_*^1}{\partial y^2}(\rho_*^0)\rho_{n,m}^0(t)\\\nonumber
&& \ \ \ \ \ \ \ \ \ \ \ \ \ \ \ \ \ -\frac{\partial^2 q_{n,m}^0}{\partial y^2}(\rho_*^0,t)\rho_*^1 - \frac{\partial q_{n,m}^1}{\partial y}(\rho_*^0,t).
\eea
Now we need to compute the terms $\frac{\partial^2 p_*^1}{\partial y^2}(\rho_*^0)$ and
$\frac{\partial q_{n,m}^1}{\partial y}(\rho_*^0,t)$ in the above equation. The method for solving the
ODEs \re{4.75}--\re{4.79} is more or less the same as the equations for the zeroth-order terms in $\tau$, except that the expressions are more complex.
We shall omit the detailed computations and summarize the final results here. One can, of course, simply
verify these results by substituting into the equations.

Here are the solutions. For the zeroth order terms that are needed in the computation of first order terms:
\beaa
&& \frac{\partial \sigma_*^0}{\partial y} =  \overline{\sigma}\frac{\sinh y}{\cosh \rho_*^0}, \mm \frac{\p \s_*^0}{\p y}(\rho_*^0)=\bar\sigma\tanh \rho_*^0, \mm
 \tilde\sigma = \frac{\bar \sigma \tanh \rho_*^0}{\rho_*^0},\\
&& \frac{\partial p_*^0}{\partial y} = \mu \tilde{\sigma} y-\mu \overline{\sigma}\frac{\sinh y}{\cosh \rho_*^0}, \mm  \frac{\partial p_*^0}{\partial y}(\rho_*^0)=0, \\
&& \frac{\partial^2 p_*^0}{\partial y^2} = \mu \tilde{\sigma} -\mu \overline{\sigma}\frac{\cosh y}{\cosh \rho_*^0}, \mm \frac{\partial^2 p_*^0}{\partial y^2}(\rho^0_*) = \mu \tilde{\sigma} -\mu \overline{\sigma} = -\mu\bar\sigma\Big(1-\frac{\tanh\rho_*^0}{\rho_*^0}\Big),\\
&& \frac{\partial^3 p_*^0}{\partial y^3} =-\mu \overline{\sigma}\frac{\sinh y}{\cosh \rho_*^0} ,\mm \frac{\partial^3 p_*^0}{\partial y^3}(\rho^0_*) = -\mu \overline{\sigma}\tanh \rho_*^0,\\
&& \frac{\p w_{n,m}^0}{\p y} = -\frac{\partial  \sigma_*^0}{\partial y}(\rho_*^0)\sqrt{1+n^{2}+m^{2}}\frac{\sinh (\sqrt{1+n^{2}+m^{2}}y)}{\cosh (\sqrt{1+n^{2}+m^{2}} \rho_*^0)}\rho_{n,m}^0(t),\\
&& \frac{\p w_{n,m}^0}{\p t} =-\frac{\partial  \sigma_*^0}{\partial y}(\rho_*^0)\frac{\cosh (\sqrt{1+n^{2}+m^{2}}y)}{\cosh (\sqrt{1+n^{2}+m^{2}} \rho_*^0)}\frac{d \rho_{n,m}^0(t)  }{d t}, \\
&& \frac{\p q_{n,m}^0}{\p y}\ = \rho_{n,m}^0(t)\bigg\{\mu\displaystyle\frac{\partial  \sigma_*^0}{\partial y}(\rho_*^0)\sqrt{1+n^{2}+m^{2}} \frac{\sinh (\sqrt{1+n^{2}+m^{2}}y) }{\cosh (\sqrt{1+n^{2}+m^{2}} \rho_*^0)}\\
&& \ \ \ \ \ \ \ \ \ \ \ \ \ + \displaystyle\frac{(n^{2}+m^{2})-2\mu\displaystyle\frac{\partial  \sigma_*^0}{\partial y}(\rho_*^0)}{2\cosh (\sqrt{n^{2}+m^{2}}\rho_*^0)}\sqrt{n^{2}+m^{2}}\sinh (\sqrt{n^{2}+m^{2}}y) \bigg\},\\
&& \frac{\p^2 q_{n,m}^0}{\p y^2} = \rho_{n,m}^0(t)\bigg\{\mu\displaystyle\frac{\partial  \sigma_*^0}{\partial y}(\rho_*^0)(1+n^{2}+m^{2}) \frac{\cosh (\sqrt{1+n^{2}+m^{2}}y) }{\cosh (\sqrt{1+n^{2}+m^{2}} \rho_*^0)}\\
&& \ \ \ \ \ \ \ \ \ \ \ \ \ + \displaystyle\frac{(n^{2}+m^{2})-2\mu\displaystyle\frac{\partial  \sigma_*^0}{\partial y}(\rho_*^0)}{2\cosh (\sqrt{n^{2}+m^{2}}\rho_*^0)}(n^{2}+m^{2})\cosh (\sqrt{n^{2}+m^{2}}y) \bigg\}.\\
\eeaa
 Now we proceed to compute the first order.
{   For $\rho_*^1$, we substitute various expressions into \re{4.45@} }and evaluate the integral to obtain:
\be \label{rho1}
\rho_*^1 =
\mu\bar{\sigma}\;\frac{- (\rho_*^0)^{2} - \rho_*^0\sinh\rho_*^0\cosh\rho_*^0+2 \sinh^{2}\rho_*^0}{2(1-\tanh \rho_*^0/\rho_*^0- \tanh^{2} \rho_*^0)  \rho_*^0  \cosh^{2}\rho_*^0      },
\ee
where the denominator is negative, by \re{2.1}.
For $p^1_*$, we only need
$ \frac{\partial^2 p_*^1}{\partial y^2}(\rho_*^0)$, {  which can be derived from \re{4.75}:}
\be
\frac{\partial^2 p_*^1}{\partial y^2} (\rho_*^0)  =  -\mu\dis\frac{\partial \sigma_*^0}{\partial y}(\rho_*^0)\frac{\partial p_*^0}{\partial y}(\rho_*^0)-\mu\sigma_*^1(\rho_*^0)
 =  -\mu\sigma_*^1(\rho_*^0) =
 \mu\overline{\sigma}\rho_*^1\tanh \rho_*^0.
 \ee

The computation of $\frac{\partial q_{n,m}^1}{\partial y}(\rho_*^0,t)$ in \re{4.80} is much more involved, but fortunately the right-hand side, the integrals and the ODEs can all be evaluated explicitly. From \re{4.77},
\beaa
&&
w_{n,m}^1 =\bigg[ -\dis\frac{\partial w_{n,m}^0}{\partial y}(\rho_*^0,t)\rho_*^1-\dis\frac{\partial \sigma_*^0}{\partial y}(\rho_*^0)\rho_{n,m}^1(t)  - \dis\frac{\partial^2 \sigma_*^0}{\partial y^2}(\rho_*^0)\rho_*^1 \rho_{n,m}^0(t)\\
&& \ \ \ \ \ \ \ \ \ \ \ \ \ -\dis\frac{\partial \sigma_*^1}{\partial y}(\rho_*^0)\rho_{n,m}^0(t)\bigg] \frac{\cosh (\sqrt{1+n^{2}+m^{2}}y)}{\cosh (\sqrt{1+n^{2}+m^{2}} \rho_*^0)}.
\eeaa
 The right-hand side of \re{4.79} is then consolidated into the form
\beaa
&  c_2(t) \sinh y \cdot \sinh(\sqrt{1+n^2+m^2}y)
 + c_3(t) \sinh y \cdot \sinh(\sqrt{n^2+m^2}y) \\
 & +  c_4(t)   y \cdot \sinh(\sqrt{1+n^2+m^2}y)
 + c_5(t) \cosh(\sqrt{1+n^2+m^2}y),
\eeaa
with
\beaa
 c_2(t) & = & -\; \frac{2\mu^2\;  \bar\s \sqrt{   1+n^2+m^2}}{\cosh \rho_*^0\cdot \cosh (\sqrt{1+n^{2}+m^{2}} \rho_*^0)}
  \; \frac{\partial  \sigma_*^0}{\partial y}(\rho_*^0) \rho_{n,m}^0(t), \\
 c_3(t) & = &  -\; \frac{\mu\;  \bar\s \sqrt{n^2+m^2}}{\cosh \rho_*^0  } \;\; \displaystyle\frac{(n^{2}+m^{2})-2\mu\displaystyle\frac{\partial  \sigma_*^0}{\partial y}(\rho_*^0)}{2\cosh (\sqrt{n^{2}+m^{2}}\rho_*^0)}  \rho_{n,m}^0(t),
 \\
  c_4(t) & = &   \frac{\mu^2\;  \tilde\s \sqrt{1+n^2+m^2}}{  \cosh (\sqrt{1+n^{2}+m^{2}} \rho_*^0)}
  \; \frac{\partial  \sigma_*^0}{\partial y}(\rho_*^0) \rho_{n,m}^0(t), \\
   c_5(t) & = & - \; \frac{\mu}{\cosh (\sqrt{1+n^{2}+m^{2}} \rho_*^0)}\bigg\{\frac{\partial  \sigma_*^0}{\partial y}(\rho_*^0) \frac{d \rho_{n,m}^0(t)  }{d t} \\
 &&   \mm -\dis\frac{\partial w_{n,m}^0}{\partial y}(\rho_*^0,t)\rho_*^1-\dis\frac{\partial \sigma_*^0}{\partial y}(\rho_*^0)\rho_{n,m}^1(t)  - \dis\frac{\partial^2 \sigma_*^0}{\partial y^2}(\rho_*^0)\rho_*^1 \rho_{n,m}^0(t)
-\frac{\partial \sigma_*^1}{\partial y}(\rho_*^0)\rho_{n,m}^0(t)\bigg\} \\
 & = & - \; \frac{\mu}{\cosh (\sqrt{1+n^{2}+m^{2}} \rho_*^0)}\bigg\{\frac{\partial  \sigma_*^0}{\partial y}(\rho_*^0) h_{n,m}(\mu,\rho_*^0) \rho_{n,m}^0(t)   \\
 &&   \mm -\dis\frac{\partial w_{n,m}^0}{\partial y}(\rho_*^0,t)\rho_*^1-\dis\frac{\partial \sigma_*^0}{\partial y}(\rho_*^0)\rho_{n,m}^1(t)  - \dis\frac{\partial^2 \sigma_*^0}{\partial y^2}(\rho_*^0)\rho_*^1 \rho_{n,m}^0(t)
-\frac{\partial \sigma_*^1}{\partial y}(\rho_*^0)\rho_{n,m}^0(t)\bigg\},
\eeaa
so that
\be
q_{n,m}^1 = \Big[c_1(t)-Q_{n,m}^1(\rho_*^0, t)\Big]\frac{\cosh (\sqrt{n^{2}+m^{2}}y)}{\cosh (\sqrt{n^{2}+m^{2}} \rho_*^0)} + Q_{n,m}^1(y,t),
\ee
where
\beaa
Q_{n,m}^1 \!\!& = \!\!&  c_2(t) \Big[ -\;\frac{ \sinh y \cdot \sinh(\sqrt{1+n^2+m^2}y)}{2(n^2+m^2) }  +
\frac{\sqrt{1+n^2+m^2}\cosh y \cdot \cosh(\sqrt{1+n^2+m^2}y)}{2(n^2+m^2) } \Big] \\
&& + c_3(t) \Big[ -\;\frac{ \sinh y \cdot \sinh(\sqrt{n^2+m^2}y)}{4(n^2+m^2)-1} +\frac{2\sqrt{n^2+m^2}\cosh y \cdot \cosh(\sqrt{n^2+m^2}y)}{4(n^2+m^2)-1  } \Big]   \\
& & +  c_4(t) \Big[  y \cdot \sinh(\sqrt{1+n^2+m^2}y) - 2\sqrt{1+n^2+m^2}  \cosh(\sqrt{1+n^2+m^2}y)\Big] \\
&& + c_5(t) \cosh(\sqrt{1+n^2+m^2}y),\\
c_1(t)\!\!&=\!\!& q_{n,m}^1(\rho_*^0,t).
\eeaa
It follows that
\beaa
\lefteqn{\frac{\p Q_{n,m}^1}{\p y}(\rho_*^0,t)\!\!} \\
& = \!\!&  c_2(t)   \;\frac{ \cosh \rho_*^0 \cdot \sinh(\sqrt{1+n^2+m^2}\rho_*^0)}{2}
 \\
&& + c_3(t) \Big[\frac{\sqrt{n^2+m^2}\sinh \rho_*^0\cdot \cosh(\sqrt{n^2+m^2}\rho_*^0)}{4(n^2+m^2)-1  }
  + \frac{(2n^2+2m^2-1)\cosh \rho_*^0 \cdot \sinh(\sqrt{n^2+m^2}\rho_*^0)}{4(n^2+m^2)-1 }  \Big]   \\
& & +  c_4(t) \Big[ \sqrt{1+n^2+m^2} \;\rho_*^0\cdot \cosh(\sqrt{1+n^2+m^2}\rho_*^0) - (1+2n^2+2m^2)  \sinh(\sqrt{1+n^2+m^2}\rho_*^0)\Big] \\
&& + c_5(t) \sqrt{1+n^2+m^2}\sinh(\sqrt{1+n^2+m^2}\rho_*^0).
\eeaa

Substituting the expressions of $c_j(t)$ into the equations, we find
\beaa
 \frac{\partial q_{n,m}^1}{\partial y}(\rho_*^0) & = &
 \Big[c_1(t)-Q_{n,m}^1(\rho_*^0, t)\Big]  \sqrt{n^{2}+m^{2}} \tanh (\sqrt{n^{2}+m^{2}} \rho_*^0)+   \frac{\p Q_{n,m}^1}{\p y}(\rho_*^0,t) \\
 & = & K_{n,m}(\mu,\rho_*^0, \rho_*^1)\rho_{n,m}^0(t)+\bigg\{\mu\overline{\sigma}\sqrt{1+n^{2}+m^{2}}\tanh \rho_*^0\tanh (\sqrt{1+n^{2}+m^{2}} \rho_*^0) \\
&&\mm +\bigg(\frac12(n^{2}+m^{2})-\mu\overline{\sigma}\tanh \rho_*^0\bigg)\sqrt{n^{2}+m^{2}}\tanh (\sqrt{n^{2}+m^{2}} \rho_*^0)\bigg\}\rho_{n,m}^1(t),
\eeaa
where
\[
 |K_{n,m}(\mu,\rho_*^0,  \rho_*^1)|\le C(n^2+m^2+1)^{5/2}.
\]

Substituting these expressions into \re{4.80}, we derive
\be \label{rho1b}
\frac{\dif \rho_{n,m}^1(t)}{\dif t}= h_{n,m}^1( \mu, \rho_*^0, \rho_*^1,  n, m) \rho_{n,m}^1(t) + k_{n,m}^1( \mu, \rho_*^0, \rho_*^1,  n, m) \rho_{n,m}^0(t),
\ee
where
\beaa
h_{n,m}^1( \mu, \rho_*^0, \rho_*^1,  n, m) & =&-\frac{\partial^2 p_*^0}{\partial y^2}(\rho_*^0)-\mu\overline{\sigma}\sqrt{1+n^{2}+m^{2}}\tanh \rho_*^0\tanh (\sqrt{1+n^{2}+m^{2}} \rho_*^0)\\
&& \ -\bigg(\frac12(n^{2}+m^{2}) -\mu\overline{\sigma}\tanh \rho_*^0\bigg)\sqrt{n^{2}+m^{2}}\tanh (\sqrt{n^{2}+m^{2}} \rho_*^0), \\
  |k_{n,m}^1( \mu, \rho_*^0, \rho_*^1,  n, m)|& \le & C (n^2+m^2+1)^{5/2}.
\eeaa

After a careful comparison with \re{4.50}, we find
\begin{equation}\label{rho1a}
 h_{n,m}^1( \mu, \rho_*^0, \rho_*^1,  n, m)= h_{n,m}(\mu,\rho_*^0).
\end{equation}
 This is significant, since the first order term in $\tau$ does not change
the leading coefficient.
If we expend the ODE for $\rho_{n,m}(t)$ in $\tau$, then
\begin{equation}
    \begin{split}\label{4.85}
    \frac{\dif \rho_{n,m}(t)}{\dif t}
    &=\frac{\dif \rho_{n,m}^0(t)}{\dif t}+\tau\frac{\dif \rho_{n,m}^1(t)}{\dif t}+O(\tau^2)\\
    &=h_{n,m}(\mu,\rho_*^0)\rho_{n,m}^0(t)+\tau h_{n,m}(\mu,\rho_*^0)\rho_{n,m}^1(t)+\tau  k_{n,m}^1 \rho_{n,m}^0(t)+O(\tau^{2})\\
    &=h_{n,m}(\mu,\rho_*^0)\rho_{n,m}(t)+\tau  k_{n,m}^1 \rho_{n,m}^0(t)+O(\tau^{2}).
    \end{split}
\end{equation}

Thus the leading order stability up to the order of $O(\tau)$ also depends on
the sign of $h_{n,m}(\mu,\rho_*^0)$. It is not difficult to see that
\be \label{nm}
 \lim_{n^2+m^2 \to \infty} \frac{h_{n,m}(\mu,\rho_*^0)}{(n^2+m^2)^{3/2}} =-\;\frac12.
\ee
We now start to investigate the sign of $h_{n,m}(\mu,\rho_*^0)$. For convenience we write (c.f., \re{4.50}),
\be \label{h_nm}
h_{n,m}(\mu,\rho_*^0) = \mu \bar\s k_1(n^2+m^2, \rho_*^0) - k_2(n^2+m^2, \rho_*^0),
\ee
where
\beaa
 k_1(j, \rho_*^0) & = & 1-\frac{\tanh \rho_*^0}{\rho_*^0}
 -\tanh \rho_*^0\cdot \Big[\sqrt{1+j}\tanh(\sqrt{1+j}\rho_*^0 )-
 \sqrt{j}\tanh(\sqrt{j}\rho_*^0)\Big] ,\\
  k_2(j, \rho_*^0) & = & \frac12 j^{3/2} \tanh(\sqrt{j}\rho_*^0).
\eeaa
From \re{2.8},  we deduce that $k_1(j,\rho_*^0)$ is monotonically increasing in $j$, and by \re{2.1}, we have
$$
k_1(0, \rho_*^0)=1- \frac{\tanh \rho_*^0}{\rho_*^0}- \tanh^{2}\rho_*^0<0, \mm
\lim_{j\rightarrow+\infty} k_1(j, \rho_*^0)=1-\frac{\tanh \rho_*^0}{\rho_*^0}>0.
$$
Hence there exists a unique $j_{0}=j_0(\rho_*^0) >0$ (not necessarily an integer) such that  $k_1(j_0,\rho_*^0)=0$. It follows that
\[
k_1(j,\rho_*^0)<0 \m\text{ for }  0\le j < j_{0}, \mm
k_1(j,\rho_*^0)>0 \m\text{ for }   j > j_{0}.
\]
Define
\begin{equation}\label{4.52}
   \mu_{j}(\rho_*^0)=\frac{  k_2({j},\rho_*^0)}{ \bar\s \; k_1({j},\rho_*^0)} \hspace{2em}\text{ for } ~ {j}> j_0.
\end{equation}
With the structure of $h_{n,m}$ given by \re{h_nm}, we established:

\begin{lem}\label{LEM4.1}
The following assertions (i)--(ii) hold:

(i) if  $0\le n^{2}+m^{2}
\le j_{0}$, then $h_{n,m}(\mu, \rho_*^0)<0$ for any $\mu>0$;

(ii) if  $j=n^{2}+m^{2}
>j_{0}$, then

\hspace{2em}(1)   $h_{n,m}(\mu, \rho_*^0)<0$ for  $0<\mu<\mu_{j}(\rho_*^0) $;

\hspace{2em}(2)  $h_{n,m}(\mu, \rho_*^0)=0$  for  $\mu=\mu_{j}(\rho_*^0)$;

\hspace{2em}(3) $h_{n,m}(\mu, \rho_*^0)>0$ for $\mu>\mu_{j}(\rho_*^0) $.
\end{lem}

For convenience we define
\be
\mu_{j}(\rho_*^0) = +\infty \m\text{for } 0\le j\le j_0, \mm
\mu_*(\rho_*^0) = \min_{j=n^2+m^2 > j_0}\mu_{j}(\rho_*^0).
\ee
We now proceed to find $ \mu_*(\rho_*^0) $.

In the following figures, we let $\bar\s=1$ and have plotted $\mu_j(\rho_*^0)$ for several  $\rho_*^0$ in the range from $\rho_*^0=0.25$ to $2$.

\renewcommand\arraystretch{1.5}
\begin{table}[ht] \label{tab1}
\begin{tabular}{c|c|c|c|c} \hlinewd{1pt}
$\rho_*^0$ & $j_0(\rho_*^0)$ & Minimum at   $j  $ & $\mu_{*}(\rho_*^0)$  &
admissibility \\ \hlinewd{1pt}
 0.25 & $47<j_0(\rho_*^0)<48 $ & $j=81$ & $\mu_{*}(\rho_*^0)=\mu_{81}(\rho_*^0)\approx 62088$
 & $81= 9^2+0^2$ \\  \hline
 0.5 & $11<j_0(\rho_*^0)<12 $ &$j=20$  & $\mu_{*}(\rho_*^0)=\mu_{20}(\rho_*^0) \approx 2088.3$
 & $20= 4^2+2^2$ \\  \hline
 $1$ &  $2<j_0(\rho_*^0)<3 $ & $j=5$   & $\mu_{*}(\rho_*^0)=\mu_5(\rho_*^0) \approx 84.054$  &  $5= 2^2+1^2$ \\  \hline
 $2$ &  $0.6<j_0(\rho_*^0)<0.7 $ & $j=1$ & $\mu_{*}(\rho_*^0)=\mu_1(\rho_*^0)\approx 5.1560$  & $1= 1^2+0^2$ \\   \hlinewd{1pt}
\end{tabular} \vspace*{1ex}
\caption{Finding $\mu_*(\rho_*^0)$}
\end{table}

Take $\rho_*^0=1$, for example, then $2<j_0(\rho_*^0)<3$, and we have plotted $\mu_j(\rho_*^0)$ for $j\ge 3${   (see Figure~3).  }The function $\mu_j(\rho_*^0)$
at non-integer values of $j$ are not plotted since they are not needed. The minimum is reached at $j=5$ and $\mu_5(1)\approx 84.054$. Since $5=2^2+1^2$, it is admissible.
Other values of $\rho_*^0$ are listed in the Table 1 and plotted in Figures 1 -- 2, 4.

 \begin{rem}
 Numerical evidence shows that $\mu_*(\rho_*^0) = \mu_1(\rho_*^0)$ when $\rho_*^0>\bar\rho\approx 1.8471$.
 \end{rem}

\begin{figure}[H] \label{fig-a}
\begin{minipage}{.48\textwidth}\centering
\includegraphics[width=2.5in]{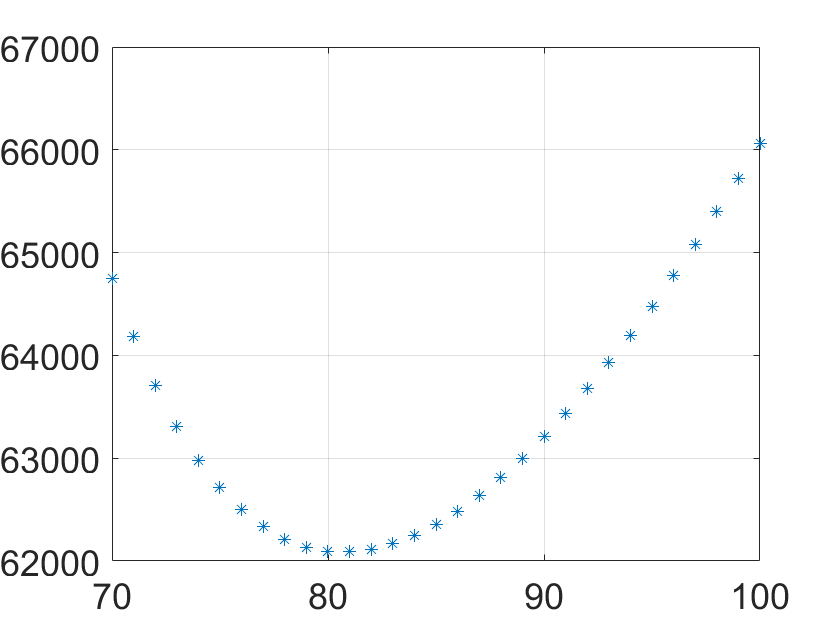}
\caption{$\mu_j(\rho_*^0)$ with $\rho_*^0=0.25$}
\end{minipage}
\begin{minipage}{.48\textwidth}\centering
\includegraphics[width=2.5in]{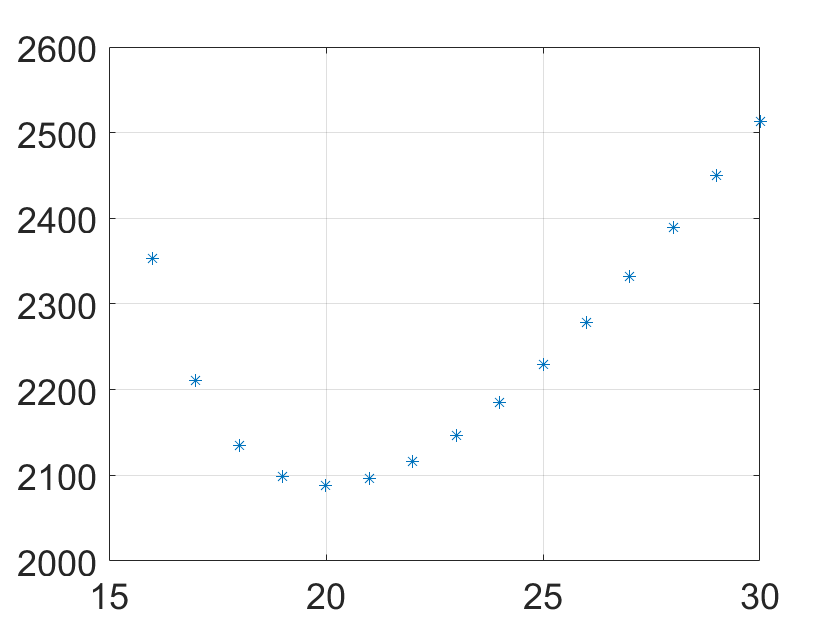}
\caption{$\mu_j(\rho_*^0)$ with $\rho_*^0=0.5$}
\end{minipage}
\end{figure}

\begin{figure}[H] \label{fig-b}
\begin{minipage}{.48\textwidth}\centering
\includegraphics[width=2.5in]{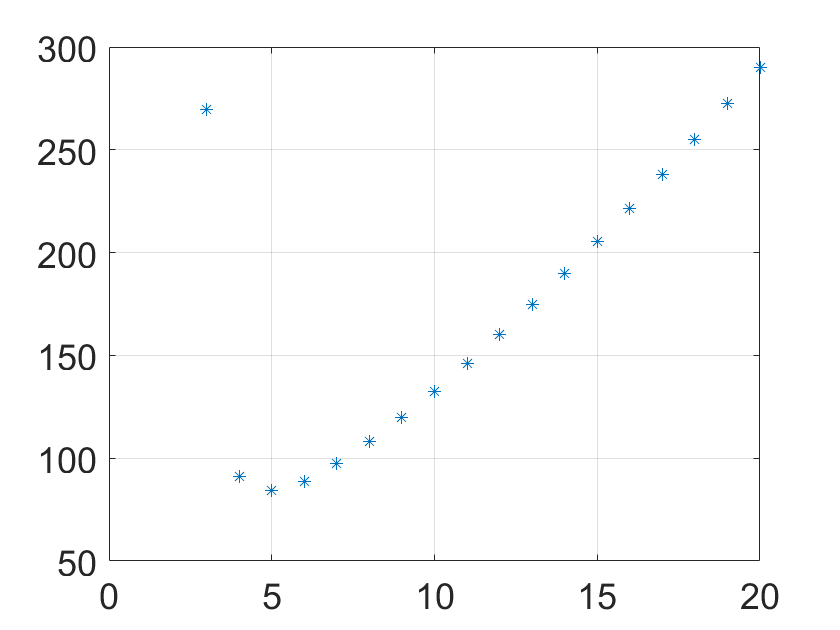}
\caption{$\mu_j(\rho_*^0)$ with $\rho_*^0=1$}
\end{minipage}
\begin{minipage}{.48\textwidth}\centering
\includegraphics[width=2.5in]{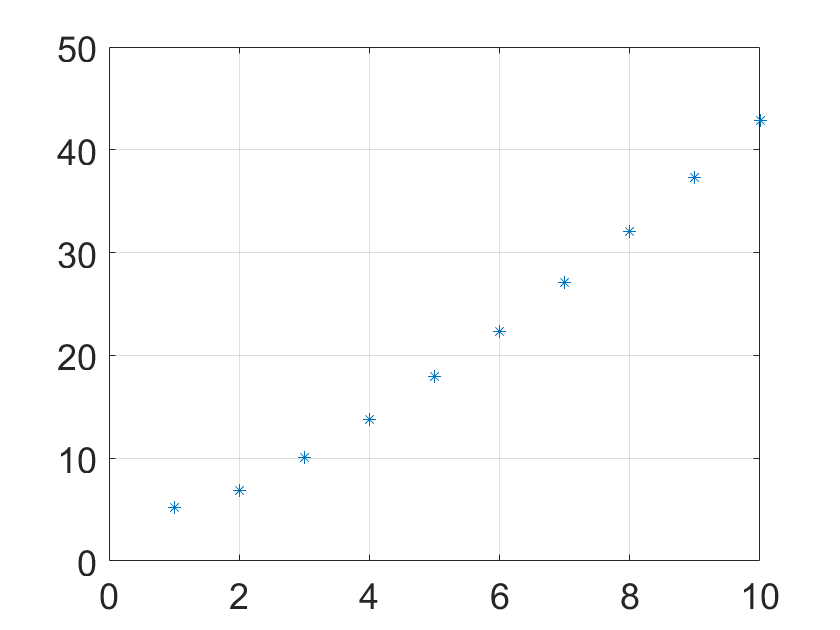}
\caption{$\mu_j(\rho_*^0)$ with $\rho_*^0=2$}
\end{minipage}
\end{figure}

\begin{lem}
If $\mu<\mu_*(\rho_*^0)$, then for sufficiently small $\tau$, there exists $\d>0$
such that
\be \label{hnm-neg}
 h_{n,m}(\mu,\rho_*^0) \le -\delta (n^2+m^2+1)^{3/2} \m\text{for all }
 n=0,1,2,\cdots, \; m = 0, 1, 2, \cdots.
\ee
\end{lem}
\begin{proof}
By \re{nm}, the estimate \re{hnm-neg} is certainly valid for $n^2+m^2>N$ for some $N$ sufficiently large. For $0\le n^2+ m^2\le N$, this estimate follows from the definition of $\mu_*(\rho_*^0)$ if $\delta$ is taken to be small enough.
\end{proof}

{\noindent \bf Proof of Theorem \ref{T1.2}. }
From the above lemma and \re{rho} we find that, for $\mu<\mu_*$,
\be
 |\rho_{n,m}^0(t)| =  |\rho_{n,m}^0(0)|  e^{h_{n,m}(\mu,\rho_*^0) t} \le
 |\rho_{n,m}^0(0)| e^{-\d (n^2+m^2+1)^{3/2}t}.
\ee
 Similarly, \re{rho1b} and \re{rho1a} imply
\be
|\rho_{n,m}^1(t)|  \le
  e^{-\d(n^2+m^2+1)^{3/2} t} \Big[  |\rho_{n,m}^1(0)| +  C|\rho_{n,m}^0(0)| (n^2+m^2+1)^{5/2}t\Big]  .
\ee
Taking Fourier series in $(n,m)$ we find that $|\rho^0(t) +\tau \rho^1(t)| \le Ce^{-\d t}$
for  $t>0 $, i.e., the system linearized both in perturbation and $\tau$ is stable.

For $\mu>\mu_*(\rho_*^0)$, the mode corresponding to the mode $(n,m)$ which gives the
minimum of $\mu_j$'s is clearly unstable.
\qed

\section{Impact of time delay}
In this section, we shall  show  the impact of time delay $\tau$ on the  tumor growth.

\begin{thm}\label{thmthmthm4.1}
$\rho_*^1 > 0$, and $\rho_*^1$ is monotonically increasing in $\mu$.
\end{thm}
\begin{proof}
By \re{rho1},
\begin{gather}
   \rho_*^1=\mu\bar{\sigma} \frac{- (\rho_*^0)^{2} - \rho_*^0\sinh\rho_*^0\cosh\rho_*^0+2 \sinh^{2}\rho_*^0}{2(1-\tanh(\rho_*^0)/\rho_*^0-\tanh^{2} \rho_*^0)  \rho_*^0  \cosh^{2}\rho_*^0      }.
\end{gather}
The denominator in the above expression is negative, by \re{2.1}. If we show that the numerator is also negative, then clearly
 $\rho_*^1 > 0$ and $\rho_*^1$ is monotonically increasing in $\mu$.

Clearly,
$
(\sinh\rho-\rho\cosh\rho)'=-\rho\sinh\rho<0$ and
$\sinh0-0\cosh0=0.$
Hence
\begin{equation}\nonumber
 f''(\rho)\triangleq  [- (\rho)^{2} - \rho\sinh\rho\cosh\rho+2 \sinh^{2}\rho]'' =4\sinh\rho(\sinh\rho-\rho\cosh\rho)<0.
\end{equation}
Since we clearly have $f'(0)=f(0)=0$, we must have $f(\rho)<0$ for $\rho>0$.
The proof is complete.
\end{proof}

\section{Conclusion}
In this paper we have investigated the impact of time delay on a tumor model in a flat domain. The existence, uniqueness, stability of the stationary problem are studied. In addition, here are some interesting observations on the impact of time delay $\tau$, the tumor aggressiveness constant $\mu$,
the nutrient supply $\bar\s$, and the tumor size (order $0$ is $\rho_*^0$, order $1$ is $\rho_*^0+\tau\rho_*^1$).

(1) Adding the time delay (at $O(\tau)$) to the system would not alter the threshold value
$\mu_*$ for which the stability of the stationary solution changes (section 4).

(2) Some other properties remain the same as in the case without time delay (see table 1):  the bigger the size (measured by $\rho_*^0$) of the tumor, the smaller the value of $\mu_*$, that is  to say that smaller stationary tumor is much more likely to be stable than its larger counter part.  When the thickness is reduced from $0.5$ to $0.25$, the tumor aggressiveness would have to increase more than 30 fold to cause problems (Table 1).  In other word, a small sized tumor is less likely to proliferate than a large sized tumor. Implication: treat the small tumor before it grows larger. Since a long time is needed for it to grow (stable), there is plenty of time to treat it.

(3) Flat perturbation (mode $(0,0)$) is always stable, regardless of the value of $\mu$.

(4) As we increase $\mu$ across the threshold of stability, the instability comes from the modes
$(0,1)$ or $(1,0)$ for larger tumors (larger $\rho_*^0)$. It can come for other modes for smaller tumors, e.g., for $\rho_*^0 = 1$, it comes from the modes $(2,1)$ or $(1,2)$ (see table 1).

(5) Adding the time delay would result in a larger stationary tumor (at $O(\tau)$) when compared to the same system without delay (Theorem 5.1). The bigger the tumor proliferation intensity $\mu$ is, the greater impact that time delay has on the size of the stationary tumor (Theorem 5.1).

(6) Since the nutrient supply $\bar\s$ appears together with $\mu$ as a product in the definition
of the threshold, increase the
nutrient supply would promote instability and proliferation.

\section*{Acknowledgements}
The research was supported by Natural Science Foundation of Guangdong, China (2018A030313523) and
Guangdong Basic and Applied Basic Research Foundation, China (2020A1515011148).

\bibliographystyle{plain}
\bibliography{main}

\begin{thebibliography}{10}

\bibitem{BVF1}
B.V. Bazaliy and A.~Friedman.
\newblock A free boundary problem for elliptic-parabolic system: application to
  a model of tumor growth.
\newblock {\em Communication in Partial Differential Equations}, 28:517--560,
  2003.

\bibitem{BVF}
B.V. Bazaliy and A.~Friedman.
\newblock Global existence and asymptotic stability for an elliptic-parabolic
  free boundary problem: an application to a model of tumor growth.
\newblock {\em Indiana University Mathematics Journal}, 52:1265--1304, 2003.

\bibitem{delay2}
H.M. Byrne.
\newblock The effect of time delays on the dynamics of avascular tumor growth.
\newblock {\em Mathematical Biosciences}, 144:83--117, 1997.

\bibitem{BC2}
H.M. Byrne and M.A.J. Chaplain.
\newblock Modelling the role of cell-cell adhesion in the growth and
  development of carcinomas.
\newblock {\em Mathematical Comput. Modeling}, 24:1--17, 1996.

\bibitem{CE1}
S.~Cui and J.~Escher.
\newblock Well-posedness and stability of a multi-dimensional tumor growth
  model.
\newblock {\em Archive for Rational Mechanics \& Analysis}, 191(1):173--193,
  2009.

\bibitem{delay4}
S.~Cui and S.~Xu.
\newblock Analysis of mathematical models for the growth of tumors with time
  delays in cell proliferation.
\newblock {\em Journal of Mathematical Analysis and Applications},
  336:523--541, 2007.

\bibitem{delay3}
U.~Forys and M.~Bodnar.
\newblock Time delays in proliferation process for solid avascular tumour.
\newblock {\em Mathematical and Computer Modelling}, 37:1201--1209, 2003.

\bibitem{FH4}
A.~Friedman and B.~Hu.
\newblock Asymptotic stability for a free boundary problem arising in a tumor
  model.
\newblock {\em Journal of Differential Equations}, 227:598--639, 2006.

\bibitem{FH3}
A.~Friedman and B.~Hu.
\newblock Bifurcation from stability to instability for a free boundary problem
  arising in a tumor model.
\newblock {\em Archive for Rational Mechanics and Analysis}, 180:293--330,
  2006.

\bibitem{FH5}
A.~Friedman and B.~Hu.
\newblock Stability and instability of liapounov-schmidt and hopf bifurcation
  for a free boundary problem arising in a tumor model.
\newblock {\em Transaction of the American Mathematical Society},
  360:5291--5342, 2008.

\bibitem{FR1}
A.~Friedman and F.~Reitich.
\newblock Analysis of a mathematical model for growth of tumor.
\newblock {\em Journal of Mathematical Biology}, 38:262--284, 1999.

\bibitem{1972Models}
H.~P. Greenspan.
\newblock Models for the growth of a solid tumor by diffusion.
\newblock {\em Studies in Applied Mathematics}, 51(4), 1972.

\bibitem{1976On}
H.~P. Greenspan.
\newblock On the growth and stability of cell cultures and solid tumors.
\newblock {\em Journal of Theoretical Biology}, 56(1):229--242, 1976.

\bibitem{2004Three}
JB. Kim, R.~Stein, and MJ. O’Hare.
\newblock Three-dimensional in vitro tissue culture models of breast cancer-- a
  review.
\newblock {\em Breast Cancer Research \& Treatment}, 85(3):281--91, 2004.

\bibitem{kyle1999characterization}
AH. Kyle, CTO. Chan, and AI. Minchinton.
\newblock Characterization of three-dimensional tissue cultures using
  electrical impedance spectroscopy.
\newblock {\em Biophysical journal}, 76(5):2640--2648, 1999.

\bibitem{2020Bifurcation}
J.~Lu and B.~Hu.
\newblock Bifurcation for a free boundary problem modeling the growth of
  multilayer tumors with ecm and mde interactions.
\newblock {\em Mathematical Methods in the Applied ences}, 43(6), 2020.

\bibitem{1997Three}
W.~Mueller-Klieser.
\newblock Three-dimensional cell cultures: from molecular mechanisms to
  clinical applications.
\newblock {\em Am.j.physiol}, 273(1):1109--23, 1997.

\bibitem{Hongjing}
H.~Pan and R.~Xing.
\newblock Bifurcation for a free boundary problem modeling tumor growth with
  ecm and mde interactions.
\newblock {\em Nonlinear Analysis: Real World Applications}, 43:362--377, 2018.

\bibitem{delay5}
S.~Xu.
\newblock Analysis of tumor growth under direct effect of inhibitors with time
  delays in proliferation.
\newblock {\em Nonlinear Analysis: Real World Applications}, 11:401--406, 2010.

\bibitem{delay6}
S.~Xu and Z.~Feng.
\newblock Analysis of a mathematical model for tumor growth under direct effect
  indirect effect of inhibitors with time delay in proliferation.
\newblock {\em Journal of Mathematical Analysis and Applications},
  374:178--186, 2011.

\bibitem{delay1}
S.~Xu, Q.~Zhou, and M.~Bai.
\newblock Qualitative analysis of a time-delayed free boundary problem for
  tumor growth under the action of external inhibitors.
\newblock {\em Mathematical Methods in the Applied Sciences}, 38:4187--4198,
  2015.

\bibitem{zhao}
X.E. Zhao and B.~Hu.
\newblock The impact of time delay in a tumor model.
\newblock {\em Nonlinear Analysis: Real World Applications}, 51:103015, 2020.

\bibitem{zhao2}
X.E. Zhao and B.~Hu.
\newblock Symmetry-breaking bifurcation for a free-boundary tumor model with
  time delay.
\newblock {\em Journal of Differential Equations}, 269:1829--1862, 2020.

\bibitem{2008Bifurcation}
F.~Zhou, J.~Escher, and S.~Cui.
\newblock Bifurcation for a free boundary problem with surface tension modeling
  the growth of multi-layer tumors.
\newblock {\em Journal of Mathematical Analysis and Applications},
  337(1):443--457, 2008.

\end{thebibliography}
\end{CJK}
\end{document}